\documentclass[english]{amsart}

\usepackage{amsmath}

\usepackage{amssymb}

\usepackage{amscd} 
\usepackage{tabularx}

\usepackage[all,cmtip,line]{xy}

\newcommand{\F}{\mathbb{F}}

\DeclareMathOperator{\Frob}{Frob}

\newcommand{\To}{\longrightarrow} 
\newcommand{\into}{\hookrightarrow} 
 
\newcommand{\isoto}{\stackrel{\sim}{\To}} 
\newcommand{\p}{\mathfrak{p}}

\newcommand{\bigO}{\mathcal{O}}

\newcommand{\R}{\mathbb{R}} 
\newcommand{\Z}{\mathbb{Z}} 
\newcommand{\A}{\mathbb{A}} 
\newcommand{\Q}{\mathbb{Q}}

\newcommand{\C}{\mathbb{C}}

\newcommand{\rhobar}{\overline{\rho}}

\newcommand{\cor}{\operatorname{cor}}

\newcommand{\Gal}{\operatorname{Gal}} 
\newcommand{\GL}{\operatorname{GL}}

\newcommand{\Qbar}{\overline{\Q}} 
\newcommand{\Qp}{\Q_p} 
\newcommand{\Ql}{\Q_l}

\newcommand{\Qlbar}{\overline{\Q}_l} 

\newcommand{\Qpbar}{\overline{\Q}_p} 
 
\newcommand{\Flbar}{\overline{\F}_l}

\newcommand{\Art}{\operatorname{Art}} 

\newcommand{\Res}{\operatorname{Res}}

\newcommand{\Ind}{\operatorname{Ind}} 
\newcommand{\SL}{\operatorname{SL}} 
\newcommand{\PSL}{\operatorname{PSL}} 
\newcommand{\ad}{\operatorname{ad}}

\newcommand{\End}{\operatorname{End}} 
\newcommand{\Hom}{\operatorname{Hom}}

\newcommand{\Sym}{\operatorname{Sym}} 
\newcommand{\rec}{\operatorname{rec}} 
\newcommand{\WD}{\operatorname{WD}} 
\newcommand{\nind}{\operatorname{n-Ind}} 
\newcommand{\Irr}{\operatorname{Irr}}
\newcommand{\WDRep}{\operatorname{WDRep}}
 
\newcommand{\Sp}{\operatorname{Sp}} 
 
\newcommand{\diag}{\operatorname{diag}}

\usepackage{amsthm}

\newtheorem*{thmn}{Theorem} 
 
\newtheorem{thm}{Theorem}[section]
\newtheorem{theorem}[thm]{Theorem}
 
\newtheorem{corollary}[thm]{Corollary} 
\newtheorem{lemma}[thm]{Lemma} 
\newtheorem{prop}[thm]{Proposition} 
\newtheorem{proposition}[thm]{Proposition} 

\newtheorem{conj}[thm]{Conjecture} \theoremstyle{definition} 
\newtheorem{defn}[thm]{Definition} \theoremstyle{remark} 
 \numberwithin{equation}{subsection}

\begin{document}

\title{The Sato-Tate Conjecture for modular forms of weight 3} 
\author{Toby Gee} \email{tgee@math.harvard.edu} \address{Department of
  Mathematics, Harvard University}\thanks{The author was partially
  supported by NSF grant DMS-0841491}

\begin{abstract} We prove a natural analogue of the Sato-Tate conjecture for modular forms of weight 2 or 3 whose associated automorphic representations are a twist of the Steinberg representation at some finite place.

\end{abstract}
\maketitle 
\tableofcontents 
\section{Introduction}The Sato-Tate conjecture is a conjecture about
the distribution of the number of points on an elliptic curve over
finite fields. Specifically, if $E$ is an elliptic curve over $\Q$
without CM, then for each prime $l$ such that $E$ has good reduction
at $l$ we
set $$a_l:=1+l-\#E(\F_l).$$ Then the Sato-Tate conjecture states that the quantities
$\cos^{-1}(a_l/2\sqrt{l})$ are equidistributed with respect to the
measure $$\frac{2}{\pi}\sin^2\theta d\theta$$ on
$[0,\pi]$. Alternatively, by the Weil bounds for $E$, the
polynomial $$X^2-a_lX+l=(X-\alpha_ll^{1/2})(X-\beta_ll^{1/2})$$
satisfies $|\alpha_l|=|\beta_l|=1$, and there is a well-defined
conjugacy class $x_{E,l}$ in $SU(2)$, the conjugacy class of the
matrix $$\begin{pmatrix} 
\alpha_l & 0\\ 0&\beta_l \end{pmatrix}.$$The Sato-Tate conjecture is
then equivalent to the statement that the classes $x_{E,l}$ are
equidistributed with respect to the Haar measure on $SU(2)$.

Tate observed that the conjecture would follow from properties of the
symmetric power 
$L$-functions  of $E$,
specifically that these $L$-functions (suitably normalised) should
have nonvanishing analytic continuation to the region $\Re s\geq
1$. This would follow  (given the modularity of elliptic curves) from the Langlands conjectures (specifically, it
would be a consequence of the symmetric power functoriality from
$\GL_2$ to $\GL_n$ for all $n$). Unfortunately, proving this functoriality appears to be
well beyond the reach of current techniques. However, Harris,
Shepherd-Baron and Taylor observed that the required analytic
properties would follow from a proof of the potential automorphy of
the symmetric power $L$-functions (that is, the automorphy of the
$L$-functions after base change to some extension of $\Q$), and were able
to use Taylor's potential automorphy techniques to prove the Sato-Tate
conjecture for all elliptic curves $E$ with non-integral
$j$-invariant (see \cite{hsbt}).

There are various possible generalisations of the Sato-Tate
conjecture; if one wishes to be maximally ambitious, one could
consider equidistribution results for the Satake parameters of rather
general automorphic representations (see for example section 2 of
\cite{MR546619langlands}). Again, such results appear to be well
beyond the range of current technology. There is, however, one special
case that does seem to be reasonable to attack, which is the case of
Hilbert cuspidal eigenforms of regular weight. In this paper, we prove
a natural generalisation of the Sato-Tate conjecture for modular
newforms (over $\Q$) of weight 2 or 3, subject to the natural analogue
of the condition that an elliptic curve has non-integral
$j$-invariant. We note that previously the only modular forms for
which the conjecture was known were those corresponding to elliptic
curves; in particular, there were no examples of weight 3 modular
forms for which the conjecture was known. After this paper was made
available, the conjecture was proved for all modular forms of weight
at least 2 in \cite{BLGHT}, by rather different methods.

Our approach is similar to that of \cite{hsbt}, and we are fortunate
in being able to quote many of their results. Indeed, it is
straightforward to check that Tate's argument
shows that the conjecture would follow from the potential automorphy
of the symmetric powers of the $l$-adic Galois representations
associated to a modular form. One might then hope to prove this
potential automorphy in the style of \cite{hsbt}; one would proceed by
realising the symmetric powers of the mod $l$ Galois representation
geometrically in such a way that their potential automorphy may be established,
and then deduce the potential automorphy of the $l$-adic
representations by means of the modularity lifting theorems of
\cite{cht} and \cite{tay06}. 

It turns out that this simple strategy encounters some significant
obstacles. First and foremost, it is an unavoidable limitation of the
known potential automorphy methods that they can only deduce that a
mod $l$ Galois representation is automorphic of minimal
weight (which we refer to as ``weight $0$''). However, the symmetric powers of the Galois representations
corresponding to modular forms of weight greater than 2 are never
automorphic of minimal weight, so one has no hope of directly proving
their potential automorphy in the fashion outlined above without some
additional argument. If, for example, one knew the weight part of
Serre's conjecture for $\GL_n$ (or even for unitary groups) one would
be able to deduce the required results, but this appears to be an
extremely difficult problem in general. There is, however, one case in
which the analysis of the Serre weights is rather easier, which is the
case that the $l$-adic Galois representations are ordinary. It is this
observation that we exploit in this paper.

In general, it is anticipated that for a given  newform $f$ of weight
$k\geq2$, there is a density one set of primes $l$ such that there
is an ordinary $l$-adic Galois representation corresponding to
$f$. Unfortunately, if $k>3$ then it is not even known that there is
an infinite set of such primes; this is the reason for our restriction to
$k=2$ or $3$. In these cases, one may use the Ramanujan conjecture and Serre's
form of the Cebotarev density theorem (see \cite{MR644559}) to prove
that the set of $l$ which are ``ordinary'' in this sense has density
one, via an argument that is presumably well-known to the experts
(although we have not been able to find the precise argument that we
use in the literature). We note that it is important for us to be able
to choose $l$ arbitrarily large in certain arguments (in order to
satisfy the hypotheses of the automorphy lifting theorems of
\cite{tay06}), so it does not appear to be possible to apply our
methods to any modular forms of weight greater than $3$. Similarly,
we cannot prove anything for Hilbert modular forms of parallel
weight $3$ over any field other than $\Q$.

We now outline our arguments in more detail, and explain exactly what
we prove. The early sections of the paper are devoted to proving the
required potential automorphy results. In section \ref{notation} we
recall some basic definitions and results from \cite{cht} on the
existence of Galois representations attached to regular automorphic
representations of $\GL_n$ over totally real and CM fields, subject to
suitable self-duality hypotheses and to the existence of finite places
at which the representations are square integrable. Section
\ref{modular forms} recalls some standard results on the Galois
representations attached to modular forms, and proves the result
mentioned above on the existence of a density one set of primes for
which there is an ordinary Galois representation.

In section \ref{in weight 0} we prove the potential automorphy in
weight $0$ of the
symmetric powers of the residual Galois representations associated to
a modular form, under the hypotheses that the residual Galois
representation is ordinary and irreducible, and the  automorphic representation
corresponding to the modular form is an unramified twist of the
Steinberg representation at some finite place. The latter condition
arises because of restrictions of our knowledge as to when there are
Galois representations associated to automorphic representations on
unitary groups, and it is anticipated that it will be possible to
remove it in the near future. That would then allow us to prove our
main theorems for any modular forms of weights $2$ or $3$ which are
not of CM type. (Note added in proof: such results are now available,
cf. \cite{shingaloisreps}, \cite{chenevierharris},
\cite{guerberoff2009modularity}, and it is thus an easy exercise to deduce
our main results without any Steinberg assumption.)

One approach to proving the potential automorphy result in weight $0$
would be to mimic the proofs for elliptic curves in \cite{hsbt}. In
fact we can do better than this, and are able to directly utilise
their results. We are reduced to proving that after making
a quadratic base change and twisting, the mod $l$ representation
attached to our modular form is, after a further base change,
congruent to a mod $l$ representation arising from a certain
Hilbert-Blumenthal abelian variety. This is essentially proved in
\cite{tay02}, and we only need to make minor changes to the proofs in
\cite{tay02} in order to deduce the properties we need. We can then
directly apply one of the main results of \cite{hsbt} to deduce the
automorphy of the even-dimensional symmetric powers of the
Hilbert-Blumenthal abelian variety, and after twisting back we deduce
the required potential automorphy of our residual
representations. Note that apart from resulting in rather clean
proofs, the advantage of making an initial congruence to a Galois
representation attached to an abelian variety and then using the
potential automorphy of the symmetric powers of this abelian variety
is that we are able to obtain local-global compatibility at all finite
places (including those dividing the residue characteristic). This
compatibility is not yet available for automorphic representations on
unitary groups in general, and is needed in our subsequent
arguments. In particular, it tells us that the automorphic
representations of weight 0 which correspond to the symmetric powers
of the $l$-adic representations coming from our Hilbert-Blumenthal
abelian variety are ordinary at $l$.

In section \ref{hida theory} we exploit this ordinarity to deduce that
the even-dimensional symmetric powers of the mod $l$ representations are potentially
automorphic of the ``correct'' weight. This is a basic consequence of
Hida theory for unitary groups, but we are not aware of any reference
that proves the precise result we need. Accordingly, we provide a
proof in the style of the arguments of \cite{taylorthesis}. There is
nothing original in this section, and as the arguments are somewhat
technical the reader may wish to skip it on a first reading. 

The results of the preceding sections are combined in section
\ref{sec:potential automorphy} to establish the required potential
automorphy results for $l$-adic (rather than mod $l$)
representations. This essentially comes down to checking the
hypotheses of the modularity lifting theorem that we wish to apply
from \cite{tay06}, which follow from the analogous arguments in
\cite{hsbt} together with the conditions that we have imposed in our
potential automorphy arguments. It is here that we need the freedom to
choose $l$ to be arbitrarily large, which results in our restriction
to weights $2$ and $3$.

Finally, in section \ref{sec:sato tate} we deduce the form of the
Sato-Tate conjecture mentioned above. As in \cite{hsbt} we have only
proved the potential automorphy of the even-dimensional symmetric
powers of the $l$-adic representations associated to our modular form,
and we deduce the required analytic properties for the $L$-functions
attached to odd-dimensional symmetric powers via an argument with
Rankin-Selberg convolutions exactly analogous to that of
\cite{hsbt}. In fact, we need to prove the same results for the
$L$-functions of certain twists of our representations by finite-order
characters, but this is no more difficult.

We now describe the form of the final result, which is slightly
different from that for elliptic curves, because our modular forms may
have non-trivial nebentypus (and indeed are required to do so if they
have weight $3$). Suppose that the newform $f$ has level $N$, nebentypus
$\chi_f$ and weight $k$; then the image of $\chi_f$ is precisely the
$m$-th roots of unity for some $m$. Then if $p\nmid N$ is a prime, we
know that
if $$X^2-a_pX+p^{k-1}\chi_f(p)=(X-\alpha_{p}p^{(k-1)/2})(X-\beta_{p}p^{{(k-1)/2}}) $$where
$a_p$ is the eigenvalue of $f$ for the Hecke operator $T_p$, then the
matrix $$\begin{pmatrix}\alpha_{p} & 0\\ 0&
  \beta_{p} \end{pmatrix}$$defines a conjugacy class $x_{f,p}$ in
$U(2)_{m}$, the subgroup of $U(2)$ of matrices with determinant an
$m$-th root of unity. Then our main result is

\begin{thmn}
  If $f$ has weight $2$ or $3$ and the associated automorphic
  representation is a twist of the Steinberg representation at some
  finite place, then the conjugacy classes $x_{f,p}$ are
  equidistributed with respect to the Haar measure on $U(2)_{m}$
  (normalised so that $U(2)_{m}$ has measure 1).
\end{thmn}
One can make this more concrete by restricting to primes $p$ such that
$\chi_f(p)$ is a specific $m$-th root of unity; see the remarks at the
end of section \ref{sec:sato tate}.

We would like to thank Thomas Barnet-Lamb, David Geraghty and Richard
Taylor for various helpful discussions during the writing of this
paper.

\section{Notation and assumptions}\label{notation} We let $\epsilon$
denote the $l$-adic cyclotomic character, regarded as a character of
the absolute Galois group of a number field or of a completion of a
number field at a finite place. We sometimes use the same notation for
the mod $l$ cyclotomic character; it will always be clear from the
context which we are referring to. We denote Tate twists in the usual
way, i.e. $\rho(n):=\rho\otimes\epsilon^n$. We write $\bar{K}$ for a
separable closure of a field $K$. If $x$ is a finite place of a number
field $F$, we will write $I_x$ for the inertia subgroup of
$\Gal(\overline{F}_x/F_x)$. We fix an algebraic closure $\Qbar$ of
$\Q$, and regard all finite extensions of $\Q$ as being subfields of
$\Qbar$. We also fix algebraic closures $\Qpbar$ of $\Qp$ for all
primes $p$, and embeddings $\Qbar\into\Qpbar$.

We need several incarnations of the local Langlands correspondence. Let $K$ be a finite extension of $\Qp$, and $l\neq p$ a prime. We have a canonical isomorphism $$\Art_K:K^\times\to W_K^{ab}$$ normalised so that geometric Frobenius elements correspond to uniformisers. Let $\Irr(\GL_n(K))$ denote the set of isomorphism classes of irreducible admissible representations of $\GL_n(K)$ over $\C$, and let $\WDRep_n(W_K)$ denote the set of isomorphism classes of $n$-dimensional Frobenius semi-simple complex Weil-Deligne  representations of the Weil group $W_K$ of $K$. The main result of \cite{ht} is that there is a family of bijections $$\rec_K:\Irr(\GL_n(K))\to\WDRep_n(W_K)$$ satisfying a number of properties that specify them uniquely (see the introduction to \cite{ht} for a complete list). Among these properties are:

\begin{itemize}
\item If $\pi\in\Irr(\GL_1(K))$ then $\rec_K(\pi)=\pi\circ\Art_K^{-1}$.
\item $\rec_K(\pi^\vee)=\rec_K(\pi)^\vee$.
\item If $\chi_1,\dots,\chi_n\in\Irr(\GL_1(K))$ are such that the normalised induction $\nind(\chi_1,\dots,\chi_n)$ is irreducible, then $$\rec_K(\nind(\chi_1,\dots,\chi_n))=\oplus_{i=1}^n\rec_K(\chi_i).$$
\end{itemize}
We will often just write $\rec$ for $\rec_K$ when the choice of $K$ is clear from the context. After choosing an isomorphism $\iota:\Qlbar\to\C$ one obtains bijections $\rec_l$ from the set of isomorphism classes of irreducible admissible representations of $\GL_n(K)$ over $\Qlbar$ to the set of isomorphism classes of $n$-dimensional Frobenius semi-simple Weil-Deligne $\Qlbar$-representations of $W_K$. We then define $r_l(\pi)$ to be the $l$-adic representation of $\Gal(\overline{K}/K)$ associated to $\rec_l(\pi^\vee\otimes|\cdot|^{(1-n)/2})$ whenever this exists (that is, whenever the eigenvalues of $\rec_l(\pi^\vee\otimes|\cdot|^{(1-n)/2})(\phi)$ are $l$-adic units, where $\phi$ is a Frobenius element). We will, of course, only use this notation where it makes sense. It is useful to note that $$r_l(\pi)^\vee(1-n)=r_l(\pi^\vee).$$

Let $M$ denote a CM field with maximal totally real subfield $F$ (by
``CM field'' we always mean ``imaginary CM field''). We denote the
nontrivial element of $\Gal(M/F)$ by $c$. Following \cite{cht} we
define a RACSDC (regular, algebraic, conjugate self dual, cuspidal) automorphic representation of $\GL_n(\A_M)$ to be a cuspidal automorphic representation $\pi$ such that
\begin{itemize}
\item $\pi^\vee\cong\pi^c$, and
\item $\pi_\infty$ has the same infinitesimal character as some irreducible algebraic representation of $\Res_{M/\Q}\GL_n$.
\end{itemize}
We say that $a\in(\Z^n)^{\Hom(M,\C)}$ is a weight if
\begin{itemize}
\item $a_{\tau,1}\geq\dots\geq a_{\tau,n}$ for all $\tau\in\Hom(M,\C)$, and
\item $a_{\tau c,i}=-a_{\tau,n+1-i}$.
\end{itemize}
For any weight $a$ we may form an irreducible algebraic representation
$W_a$ of $\GL_n^{\Hom(M,\C)}$, the tensor product over $\tau$ of the
irreducible algebraic representations of $\GL_n$ with highest weight
$a_\tau$. We say that $\pi$ has weight $a$ if it has the same
infinitesimal character as $W_a^\vee$; note that any RACSDC
automorphic representation has some weight.  Let $S$ be a non-empty
finite set of finite places of $M$. For each $v\in S$, choose an
irreducible square integrable representation $\rho_v$ of $\GL_n(M_v)$
(in this paper, we will in fact only need to consider the case where
each $\rho_v$ is the Steinberg representation). We say that an RACSDC
automorphic representation $\pi$ has type $\{\rho_v\}_{v\in S}$ if for
each $v\in S$, $\pi_v$ is an unramified twist of $\rho_v^\vee$. There
is a compatible family of Galois representations associated to such a
representation in the following fashion.

\begin{prop}
  Let $\iota:\Qlbar\isoto\C$. Suppose that $\pi$ is an RACSDC
  automorphic representation of $\GL_n(\A_M)$ of type $\{\rho_v\}_{v\in S}$ for some nonempty set of finite places $S$. Then there is a continuous semisimple representation $$r_{l,\iota}(\pi):\Gal(\overline{M}/M)\to\GL_n(\Qlbar)$$ such that
  \begin{enumerate}
  \item For each finite place $v\nmid l$ of $M$, we have $$r_{l,\iota}(\pi)|_{\Gal(\overline{M}_v/M_v)}^{ss}=r_l(\iota^{-1}\pi_v)^\vee(1-n)^{ss}.$$
\item $r_{l,\iota}(\pi)^c=r_{l,\iota}(\pi)^\vee\epsilon^{1-n}$.
  \end{enumerate}

\end{prop}
\begin{proof}
  This follows from Proposition 4.2.1 of \cite{cht} (which in fact also gives information on $r_{l,\iota}|_{\Gal(\overline{M}_v/M_v)}$ for places $v|l$).
\end{proof}
The representation $r_{l,\iota}(\pi)$ may be conjugated to be valued in the
ring of integers of a finite extension of $\Ql$, and we may reduce it
modulo the maximal ideal of this ring of integers and semisimplify to
obtain a well-defined continuous
representation $$\bar{r}_{l,\iota}(\pi):\Gal(\overline{M}/M)\to\GL_n(\Flbar).$$

Let $a\in(\Z^n)^{\Hom(M,\Qlbar)}$, and let
$\iota:\Qlbar\isoto\C$. Define $\iota_*a\in(\Z^n)^{\Hom(M,\C)}$ by
$(\iota_*a)_{\iota\tau,i}=a_{\tau,i}$. Now let $\rho_v$ be a discrete
series representation of $\GL_n(M_v)$ over $\Qlbar$ for each $v\in S$.
If $r:\Gal(\overline{M}/M)\to\GL_n(\Qlbar)$, we say that $r$ is
automorphic of weight $a$ and type $\{\rho_v\}_{v\in S}$ if $r\cong
r_{l,\iota}(\pi)$ for some RACSDC automorphic representation $\pi$ of
weight $\iota_*a$ and type $\{\iota\rho_v\}_{v\in S}$. Similarly, if
$\bar{r}:\Gal(\overline{M}/M)\to\GL_n(\Flbar)$, we say that $\bar{r}$  is
automorphic of weight $a$ and type $\{\rho_v\}_{v\in S}$ if $\bar{r}\cong
\bar{r}_{l,\iota}(\pi)$ for some RACSDC automorphic representation
$\pi$ with $\pi_l$ unramified, of weight $\iota_*a$ and type $\{\iota\rho_v\}_{v\in S}$.

We now consider automorphic representations of $\GL_n(\A_F)$. We say
that a cuspidal automorphic representation $\pi$ of $\GL_n(\A_F)$ is
RAESDC (regular, algebraic, essentially self dual, cuspidal) if
\begin{itemize}
\item $\pi^\vee\cong\chi\pi$ for some character $\chi:F^\times\backslash\A_F^\times\to\C^\times$ with $\chi_v(-1)$ independent of $v|\infty$, and
\item $\pi_\infty$ has the same infinitesimal character as some
  irreducible algebraic representation of $\Res_{F/\Q}\GL_n$.
\end{itemize}
We say that $a\in(\Z^n)^{\Hom(F,\C)}$ is a weight if
$$a_{\tau,1}\geq\dots\geq a_{\tau,n}$$ for all
$\tau\in\Hom(F,\C)$. For any weight $a$ we may form an irreducible
algebraic representation $W_a$ of $\GL_n^{\Hom(F,\C)}$, the tensor
product over $\tau$ of the irreducible algebraic representations of
$\GL_n$ with highest weight $a_\tau$. We say that an RAESDC
automorphic representation $\pi$ has weight $a$ if it has the same
infinitesimal character as $W_a^\vee$. In this case, by the
classification of algebraic characters over a totally real field, we
must have $a_{\tau,i}+a_{\tau,n+1-i}=w_a$ for some $w_a$ independent
of $\tau$. Let $S$ be a non-empty finite set of finite places of
$F$. For each $v\in S$, choose an irreducible square integrable
representation $\rho_v$ of $\GL_n(M_v)$. We say that an RAESDC
automorphic representation $\pi$ has type $\{\rho_v\}_{v\in S}$ if for
each $v\in S$, $\pi_v$ is an unramified twist of $\rho_v^\vee$. Again,
there is a compatible family of Galois representations associated to
such a representation in the following fashion.
\begin{prop}\label{existence of galois reps for raesdc}
  Let $\iota:\Qlbar\isoto\C$. Suppose that $\pi$ is an RAESDC
  automorphic representation of $\GL_n(\A_F)$, of type
  $\{\rho_v\}_{v\in S}$ for some nonempty set of finite places $S$,
  with $\pi^\vee\cong\chi\pi$. Then there is a continuous semisimple
  representation $$r_{l,\iota}(\pi):\Gal(\overline{F}/F)\to\GL_n(\Qlbar)$$
  such that
  \begin{enumerate}
  \item For each finite place $v\nmid l$ of $F$, we have $$r_{l,\iota}(\pi)|_{\Gal(\overline{F}_v/F_v)}^{ss}=r_l(\iota^{-1}\pi_v)^\vee(1-n)^{ss}.$$
\item $r_{l,\iota}(\pi)^\vee=r_{l,\iota}(\chi)\epsilon^{n-1}r_{l,\iota}(\pi)$.
  \end{enumerate}Here $r_{l,\iota}(\chi)$ is the $l$-adic Galois representation associated to $\chi$ via $\iota$ (see Lemma 4.1.3 of \cite{cht}).

\end{prop}
\begin{proof}
  This is Proposition 4.3.1 of \cite{cht} (which again obtains a
  stronger result, giving information on
  $r_{l,\iota}|_{\Gal(\overline{F}_v/F_v)}$ for places $v|l$).
\end{proof}
Again, the representation $r_{l,\iota}(\pi)$ may be conjugated to be valued in the
ring of integers of a finite extension of $\Ql$, and we may reduce it
modulo the maximal ideal of this ring of integers and semisimplify to
obtain a well-defined continuous
representation $$\bar{r}_{l,\iota}(\pi):\Gal(\overline{F}/F)\to\GL_n(\Flbar).$$

Let $a\in(\Z^n)^{\Hom(F,\Qlbar)}$, and let
$\iota:\Qlbar\isoto\C$. Define $\iota_*a\in(\Z^n)^{\Hom(F,\C)}$ by
$(\iota_*a)_{\iota\tau,i}=a_{\tau,i}$. Let $\rho_v$ be a discrete
series representation of $\GL_n(M_v)$ over $\Qlbar$ for each $v\in S$. If
$r:\Gal(\overline{F}/F)\to\GL_n(\Qlbar)$, we say that $r$ is
automorphic of weight $a$ and type $\{\rho_v\}_{v\in S}$ if $r\cong
r_{l,\iota}(\pi)$ for some RAESDC automorphic representation $\pi$ of
weight $\iota_*a$ and type $\{\iota\rho_v\}_{v\in S}$. Similarly, if
$\bar{r}:\Gal(\overline{F}/F)\to\GL_n(\Flbar)$, we say that $\bar{r}$
is automorphic of weight $a$ and type $\{\rho_v\}_{v\in S}$ if
$\bar{r}\cong \bar{r}_{l,\iota}(\pi)$ for some RAESDC automorphic
representation $\pi$ with $\pi_l$ unramified, of weight $\iota_*a$ and type
$\{\iota\rho_v\}_{v\in S}$.

As in \cite{hsbt} we denote the Steinberg representation of
$\GL_n(K)$, $K$ a nonarchimedean local field, by $\Sp_n(1)$.

\section{Modular forms}\label{modular forms}\subsection{}Let $f$ be a
cuspidal newform of level $\Gamma_{1}(N)$, nebentypus $\chi_f$, and
weight $k\geq 2$. Suppose that for each prime $p\nmid N$ we have
$T_{p}f=a_{p}f$. Then each $a_{p}$ is an algebraic integer, and the
set $\{a_{p}\}$ generates a number field $K_{f}$ with ring of integers
$\bigO_{f}$. We will view $K_{f}$ as a subfield of $\C$. It is known
that $K_{f}$ contains the image of $\chi_f$. For each place
$\lambda|l$ of $\bigO_{f}$ there is a continuous
representation $$\rho_{f,\lambda}:\Gal(\Qbar/\Q)\to\GL_{2}(K_{f,\lambda}) $$
which is determined up to isomorphism by the property that for all
$p\nmid Nl$, $\rho_{f,\lambda}|_{\Gal(\Qpbar/\Qp)}$ is unramified, and
the characteristic polynomial of $\rho_{f,\lambda}(\Frob_{p})$ is
$X^{2}-a_{p}X+p^{k-1}\chi_f(p)$ (where $\Frob_p$ is a choice of a
geometric Frobenius element at $p$).

 Assume from now on that $f$ is not of CM type.

\begin{defn}
	Let $\lambda$ be a prime of $\bigO_{f}$ lying over a rational prime $l$. Then we say that $f$ is ordinary at $\lambda$ if $\lambda\nmid a_{l}$. We say that $f$ is ordinary at $l$ if it is ordinary at $\lambda$ for some $\lambda|l$.
\end{defn}

\begin{lemma}\label{lem:weight two or three is ordinary}If $k=2$ or $3$, then the set of primes $l$ such that $f$ is ordinary at $l$ has density one.
	
\end{lemma}
\begin{proof}The proof is based on an argument of Wiles (see the final lemma of \cite{wil90}). Let $S$ be the finite set of primes which either divide $N$ or which are ramified in $\bigO_{f}$. Suppose that $f$ is not ordinary at $p\notin S$. By definition we have that $\lambda|a_{p}$ for each prime $\lambda$ of $\bigO_{f}$ lying over $p$. Since $p$ is unramified in $\bigO_{f}$, $(p)=\prod_{\lambda|p}\lambda$, so $p|a_{p}$. Write $a_{p}=pb_{p} $ with $b_{p}\in\bigO_{f}$.
	
Since $p\nmid N$, the Weil bounds (that is, the Ramanujan-Petersson conjecture) tell us that for each embedding $\iota:K_{f}\into\C$ we have $|\iota(a_{p})|\leq 2p^{(k-1)/2}$. Since $k\leq 3$, this implies that $|\iota(b_{p})|\leq 2$ for all $\iota$. Let $T$ be the set of $y\in\bigO_{f}$ such that $|\iota(y)|\leq 2$ for all $\iota$. This is a finite set, because one can bound the absolute values of the coefficients of the characteristic polynomial of such a $y$.

From the above analysis, it is sufficient to prove that for each $y\in T$, the set of primes $p$ for which $a_{p}=py$ has density zero. However, by Corollaire 1 to Th\'{e}or\`{e}me 15 of \cite{MR644559}, the number of primes $p\leq x$ for which $a_{p}=py$ is $O(x/(\log x)^{5/4-\delta})$ for any $\delta>0$, which immediately shows that the density of such primes is zero, as required.\end{proof}
The following result is well known, and follows from, for example, \cite{MR1047142} and Theorem 2 of \cite{wil90}.
\begin{lemma}\label{lem:existence of ordinary Galois representation} If $f$ is ordinary at a place $\lambda|l$ of $\bigO_{f}$, and $l\nmid N$, then the Galois representation $\rho_{f,\lambda}$ is crystalline, and furthermore it is ordinary; that is, $$\rho_{f,\lambda}|_{\Gal(\Qlbar/\Ql)}\cong \begin{pmatrix} \psi_{1} & *\\ 0& \psi_{2}\epsilon^{1-k} \end{pmatrix} $$where $\psi_{1}$ and $\psi_{2}$ are unramified characters of finite order. In addition, $\psi_{1}$ takes $\Frob_{l}$ to the unit root of $X^{2}-a_{l}X+\chi_f(l)l^{k-1}$.
\end{lemma}

\subsection{}Let $\rhobar_{f,\lambda}$ denote the semisimplification of the reduction mod $\lambda$ of $\rho_{f,\lambda}$; this makes sense because $\rho_{f,\lambda}$ may be conjugated to take values in $\GL_2(\bigO_{f,\lambda})$, and it is independent of the choice of lattice. It is valued in $\GL_2(k_{f,\lambda})$, where $k_{f,\lambda}$ is the residue field of $K_{f,\lambda}$. \begin{defn} We say that $\rhobar_{f,\lambda}$ has \emph{large image} if $$\SL_{2}(k)\subset\rhobar_{f,\lambda}(\Gal(\Qbar/\Q))\subset k_{f,\lambda}^{\times}\GL_{2}(k)  $$for some subfield $k$ of $k_{f,\lambda}$.\end{defn}
We will need to know that the residual Galois representations $\rhobar_{f,\lambda}$ frequently have large image. The following result is essentially due to Ribet (see \cite{MR0419358}, which treats the case $N=1$; for a concrete reference, which also proves the corresponding result for Hilbert modular forms, see \cite{MR2172950}).

\begin{lemma}\label{lem:modular form large image}For all but finitely many primes $\lambda$ of $\bigO_{f}$,  $\rhobar_{f,\lambda}$ has large image.
	
\end{lemma}

\subsection{}We let $\pi(f)$ be the automorphic representation of
$\GL_{2}(\A_{\Q})$ corresponding to $f$, normalised so that  $\pi(f)$
is RAESDC of weight $(k-2,0)$ (it is essentially self dual
because $$\pi(f)^{\vee}\cong\chi\pi(f) $$where
$\chi=|\cdot|^{k-2}\chi_f^{-1}$). Let
$\lambda|l$ be a place of $\bigO_{f}$, and choose an isomorphism
$\iota:\Qbar_{l}\isoto\C$ and a compatible embedding
$K_{f,\lambda}\into\Qbar_{l}$; that is, an embedding such that the
diagram $$\xymatrix{K_f\ar[r]\ar[d] & \C\\ K_{f,\lambda}\ar[r] &
  \Qlbar\ar[u]_\iota}$$commutes.  Assume that $\pi_{f,v}$ is square
integrable for some finite place $v$. Then by Proposition \ref{existence of galois reps for raesdc} there is a Galois representation $$r_{l,\iota}(\pi(f)):\Gal(\Qbar/\Q)\to\GL_{2}(\Qbar_l) $$ associated to $\pi_{f}$, and it follows from the definitions that $$r_{l,\iota}(\pi(f))\cong \rho_{f,\lambda}\otimes_{K_{f,\lambda}}\Qbar_{l}. $$

\begin{defn}
  We say that $f$ is Steinberg at a prime $q$ if $\pi(f)_q$ is an unramified twist of the Steinberg representation.
\end{defn}

\begin{defn} We say that $f$ is potentially Steinberg at a prime $q$
  if $\pi(f)_q$ is a (possibly ramified) twist of the Steinberg representation.\end{defn}

Note that if $f$ is (potentially) Steinberg at $q$ for some $q$ then it is not CM. Note also that if $f$ is potentially Steinberg at $q$ then there is a Dirichlet character $\theta$ such that $f\otimes\theta$ is Steinberg at $q$.

\section{Potential automorphy in weight 0}\label{in weight
  0}\subsection{}Let $l$ be an odd prime, and let $f$ be a modular
form of weight $2\leq k<l$ and level $N$, $l\nmid N$. Assume that $f$ is
Steinberg at $q$. Suppose that $\lambda|l$ is a place of $\bigO_{f}$
such that $f$ is ordinary at $\lambda$. Assume that
$\rhobar_{f,\lambda}$ is absolutely irreducible. By Lemma
\ref{lem:existence of ordinary Galois
  representation} we have $$\rhobar_{f,\lambda}|_{\Gal(\Qlbar/\Ql)}\cong \begin{pmatrix}
  \overline{\psi}_{1} & *\\ 0&
  \overline{\psi}_{2}\epsilon^{1-k} \end{pmatrix} $$where
$\overline{\psi}_{1}$ and $\overline{\psi}_{2}$ are unramified
characters. We wish to prove that the symmetric powers of
$\rhobar_{f,\lambda}$ are potentially automorphic of some weight. To
do so, we use a potential modularity argument to realise
$\rhobar_{f,\lambda}$ geometrically, and then appeal to the results of
\cite{hsbt}.

The potential modularity result that we need is almost proved in \cite{tay02}; the one missing ingredient is that we wish to preserve the condition of being Steinberg at $q$. This is, however, easily arranged, and rather than repeating all of the arguments of \cite{tay02}, we simply indicate the modifications required.

We begin by recalling some definitions from \cite{tay02}. Let $N$ be a totally real field. Then an $N$-HBAV over a field $K$ is a triple $(A,i,j)$ where \begin{itemize}
	\item $A/K$ is an abelian variety of dimension $[N:\Q]$,
	\item $i:\bigO_{N}\into\End(A/K)$, and
	\item $j:\bigO_{N}^{+}\isoto\mathcal{P}(A,i)$ is an isomorphism of ordered invertible $\bigO_{N}$-modules.
\end{itemize}For the definitions of ordered invertible $\bigO_{N}$-modules and of $\bigO_{N}^{+}$\ and $\mathcal{P}(A,i)$, see page 133 of \cite{tay02}. 

Choose a totally real quadratic field $F$ in which $l$ is inert and
$q$ is unramified and which is linearly disjoint from
$\Qbar^{\ker(\rhobar_{f,\lambda})}$ over $\Q$, a finite extension
$k/k_{f,\lambda}$ and a character
$\overline{\theta}:\Gal(\overline{F}/F)\to k^{\times}$ which is
unramified at $q$ such
that $$\det\rhobar_{f,\lambda}|_{\Gal(\overline{F}/F)}=\epsilon^{-1}\overline{\theta}^{-2}$$and
$(\rhobar_{f,\lambda}|_{\Gal(\overline{F}/F)}\otimes\overline{\theta})(\Frob_w)$
has eigenvalues $1$, $\#k(w)$, where $w|q$ is a place of $F$. This is
possible as the obstruction to taking a square root of a character is
in the 2-part of the Brauer group, and because any class in the Brauer
group of a local field splits over an unramified extension. Let
$\rhobar=\rhobar_{f,\lambda}|_{\Gal(\overline{F}/F)}\otimes\overline{\theta}:\Gal(\overline{F}/F)\to\GL_2(k)$,
so that $\det\rhobar=\epsilon^{-1}$. If $x$ is the place of $F$ lying
over $l$, then we may write (for some character $\overline{\chi}_x$ of
$\Gal(\overline{F_x}/F_x)$) $$\rhobar|_{\Gal(\overline{F_x}/F_x)}\cong \begin{pmatrix}
  \overline{\chi}_x^{-1} & *\\ 0&
  \overline{\chi}_x\epsilon^{-1} \end{pmatrix} $$with
$\overline{\chi}_x^{2}|_{I_{x}}=\epsilon^{2-k}$.

\begin{thm}\label{thm:potential modularity: HBAV for rhobar}There is a
  finite totally real Galois extension $E/F$ which is linearly
  disjoint from $\Qbar^{\ker(\rhobar_{f,\lambda})}$ over $\Q$ and in
  which the unique prime of $F$ dividing $l$ splits completely, a totally real field $N$, a place
  $\lambda'|l$ of $N$, a place $v_{q}|q$ of $E$, and an $N$-HBAV
  $(A,i,j)/E$ with potentially good reduction at all places dividing
  $l$ such that \begin{itemize}
	\item the representation of $\Gal(\bar{E}/E)$ on $A[\lambda']$ is equivalent to $(\rhobar|_{\Gal(\bar{E}/E)})^\vee$,
	\item at each place $x|l$ of $E$, the action of $\Gal(\overline{E_x}/E_x)$ on $T_{\lambda'}A\otimes\Ql$ is of the form $$ \begin{pmatrix} {\chi}_x^{-1}\epsilon & *\\ 0& {\chi}_x \end{pmatrix} $$with $\chi_x$ a tamely ramified lift of $\overline{\chi}_x$, and
	\item $A$ has multiplicative reduction at $v_{q}$.
\end{itemize}
	
\end{thm}
\begin{proof}As remarked above, this is essentially proved in \cite{tay02}. Indeed, if $k>2$ then with the exception of the fact that $E$ can be chosen to be linearly disjoint from $\Qbar^{\ker(\rhobar_{f,\lambda})}$ over $\Q$, and the claim that $A$ can be chosen to have multiplicative reduction at some place over $q$, the result is obtained on page 136 of \cite{tay02} (the existence of $A$ with $A[\lambda']$ equivalent to $(\rhobar|_{\Gal(\bar{E}/E)})^\vee$ is established in the second paragraph on that page, and the form of the action of $\Gal(\overline{E_x}/E_x)$ for $x|l$ follows from Lemma 1.5 of \emph{loc. cit.} ).
	
  We now indicate the modifications needed to the arguments of
  \cite{tay02} to obtain the slight strengthening that we
  require. Suppose firstly that $k>2$. Rather than employing the
  theorem of Moret-Bailly stated as Theorem G of \cite{tay02}, we use
  the variant given in Proposition 2.1 of \cite{hsbt}. This
  immediately allows us to assume that $E$ is linearly disjoint from
  $\Qbar^{\ker(\rhobar_{f,\lambda})}$ over $\Q$, so we only need to
  ensure that $A$ has multiplicative reduction at some place dividing
  $q$. Let $X$ be the moduli space defined in the first paragraph of
  page 136 of \cite{tay02}. Let $v$ be a place of $F$ lying over
  $q$. It is enough to check that there is a non-empty open subset
  $\Omega_{v}$ of $X(F_{v})$ such that for each point of $\Omega_{v}$,
  the corresponding $N$-HBAV has multiplicative reduction. Let
  $\Omega_{v}$ denote the set of \emph{all} points of $X({F}_{v})$
  such that the corresponding $N$-HBAV has multiplicative reduction;
  this is an open subset of $X({F}_{v})$, and it is non-empty (by the
  assumptions on $\bar{\theta}$ at places of $F$ dividing $q$, and the
  assumption that $\pi(f)$ is an unramified twist of the Steinberg
  representation, we see that $\rhobar(\Frob_v)$ has eigenvalues $1$
  and $\#k(v)$, and is congruent to a Galois representation attached
  to an unramified twist of a Steinberg representation, so any
  $N$-HBAV with multiplicative reduction suffices), as required.
	
        If $k=2$, then the only additional argument needed is one to
        ensure that if $\overline{\chi}_x^2=1$, then the abelian
        variety can be chosen to have good reduction rather than
        multiplicative reduction. This follows easily from the fact
        that $\rhobar|_{\Gal(\overline{F_x}/F_x)}$ is finite flat 
        (cf. the proof of Theorem 2.1 of \cite{kw}, which establishes
        a very similar result).
\end{proof}

Let $M$ be a totally real field, and let $(A,i,j)/M$ be an
$N$-HBAV. Fix an embedding $N\subset \R$. We recall some definitions
from section 4 of \cite{hsbt}. For each finite place $v$ of $M$ there
is a two dimensional Weil-Deligne representation $\WD_v(A,i)$ defined
over $\overline{N}$ such that if $\mathfrak{p}$ is a place of $N$ of
residue characteristic $p$ different from the residue characteristic
of $v$, we
have $$\WD(H^1(A\times\overline{M},\Qp)|_{\Gal(\overline{M}_v/M_v)}\otimes_{N_p}\overline{N}_\mathfrak{p})\cong\WD_v(A,i)\otimes_{\overline{N}}\overline{N}_\p.$$

\begin{defn}\label{defn:sym of abelian variety is automorphic}We say
  that $\Sym^{m}A$ is automorphic of type $\{\rho_{v}\}_{v\in S}$ if
  there is an RAESDC representation $\pi$ of $\GL_{m+1}(\A_{M})$ of
  weight 0 and type $\{\rho_{v}\}_{v\in S}$ such that for all finite
  places $v$ of
  $M$, $$\rec(\pi_{v})|\Art_{M_v}^{-1}|^{-m/2}=\Sym^{m}\WD_{v}(A,i).$$\end{defn}

\begin{theorem}\label{thm:potential modularity in weight 0 for sym of
    ab var}Let $E$, $A$ be as in the statement of Theorem
  \ref{thm:potential modularity: HBAV for rhobar}. Let $\mathcal{N}$
  be a finite set of even positive integers. Then there is a finite Galois totally real extension $F'/E$ and a place $w_{q}|q$ of $F'$ such that \begin{itemize}
	\item for any $n\in\mathcal{N}$, $\Sym^{n-1}A$ is automorphic over $F'$ of weight 0 and type $\{\Sp_{n}(1)\}_{\{w_{q}\}}$,
	\item The primes of $E$ dividing $l$ are unramified in $F'$, and 
	\item $F'$ is linearly disjoint from $\Qbar^{\ker({\rhobar_{f,\lambda}})}$ over $\Q$.
\end{itemize} 
	
\end{theorem}
\begin{proof}This is essentially Theorem 4.1 of \cite{hsbt}. In
  particular, the proof in \cite{hsbt} establishes that there is a
  Galois totally real extension $F'/E$, and a place $w_{q}$ of $F'$
  lying over $q$ such that for any $n\in\mathcal{N}$, $\Sym^{n-1}A$ is
  automorphic over $F'$ of weight 0 and type
  $\{\Sp_{n}(1)\}_{\{w_{q}\}}$. Note that the $l$ used in their
  argument is \emph{not} the $l$ used here. To complete the proof, we
  need to establish that it is possible to obtain an $F'$ in which $l$
  is unramified, and which is linearly disjoint from
  $\Qbar^{\ker({\rhobar_{f,\lambda}})}$ over $\Q$. The latter point
  causes no difficulty, but the first point requires some minor
  modifications of the arguments of \cite{hsbt}. We now outline the
  necessary changes.

  To aid comparison to \cite{hsbt}, for the rest of this proof we will
  refer to our $l$ as $s$; all references to $l$ will be to primes of
  that name in the proofs of various theorems in \cite{hsbt}. We begin
  by choosing a finite solvable totally real extension $L$ of $E$, linearly
  disjoint from $\Qbar^{\ker({\rhobar_{f,\lambda}})}$ over $\Q$, such
  that the base change of $A$ to $L$ has good reduction at all places
  dividing $s$. Choose a prime $l$ as in the proof of Theorem 4.1 of
  \cite{hsbt}. We then apply a slight modification of Theorem 4.2 of \emph{loc.cit.},
  with the conclusion strengthened to include the hypothesis that $s$
  is unramified in $F'$. To prove this, in the proof of Theorem 4.2 of
  \emph{loc.cit.}, note that $F_1=E$. Choose all auxiliary primes not
  to divide $s$. Rather than constructing a
  moduli space $X_W$ over $E$, construct the analogous space over $L$,
  and consider the restriction of scalars $Y=\Res_{L/E}(X_W)$. Applying
  Proposition 2.1 of \cite{hsbt} to $Y$, rather than $X_W$, we may
  find a finite totally real Galois extension $F^{(1)}/E$ in which $s$
  is unramified, such that $Y$ has an $F^{(1)}$-point. Furthermore, we
  may assume that $F^{(1)}$ is linearly
  disjoint from $\Qbar^{\ker({\rhobar_{f,\lambda}})}$ over $\Q$. Note
  that an $F^{(1)}$-point of $Y$ corresponds to an $F^{(1)}L$-point of
  $X_W$.

We now make a similar modification to the proof of Theorem 3.1 of
\cite{hsbt}, replacing the schemes $T_{W_i}$ over $F$ with
$\Res_{LF/F}T_{W_i}$. We conclude that there is a finite Galois
totally real extension $F'/E$ in which $s$ is unramified, which is linearly
  disjoint from $\Qbar^{\ker({\rhobar_{f,\lambda}})}$ over $\Q$, such
  that for any $n\in\mathcal{N}$, $\Sym^{n-1}A$ is automorphic over
  $F'L$ of weight 0 and type $\{\Sp_{n}(1)\}_{\{w_{q}\}}$. Since the
  extension $F'L/F'$ is solvable, it follows from solvable base change
  (e.g. Lemma 1.3 of \cite{BLGHT}) that in fact for any $n\in\mathcal{N}$, $\Sym^{n-1}A$ is automorphic over
  $F'$ of weight 0 and type $\{\Sp_{n}(1)\}_{\{w_{q}\}}$, as required.\end{proof}We may
now twist $\rhobar$ by $\overline{\theta}^{-1}$ in order to deduce
results about $\rhobar_{f,\lambda}$. Let $N$ and $\lambda$ be as in
the statement of Theorem \ref{thm:potential modularity: HBAV for
  rhobar}. Fix an embedding $N_{\lambda'}\into\Qlbar$. Let $\theta$ be
the Teichm\"uller lift of $\overline{\theta}$, and let $\rho_{n}$
denote the action of $\Gal(\bar{E}/E)$
on $$\Sym^{n-1}(H^1(A\times\overline{E},\Ql)\otimes_{N_l}N_{\lambda'}\otimes\theta^{-1})\otimes_{N_\lambda'}\Qlbar.$$
By construction, $\rho_{n}$ is a lift of
$\Sym^{n-1}\rhobar_{f,\lambda}|_{\Gal(\overline{E}/E)}\otimes_{k_{f,\lambda}}\bar{\F}_l$
(where the embedding $k_{f,\lambda}\into\Flbar$ is determined by the
embedding $k\into\Flbar$ induced by the embedding
$N_{\lambda'}\into\Qlbar$). Note also that (again by construction) at
each place $x|l$ of
$E$, $$\rho_{2}|_{\Gal(\overline{E_x}/E_x)}\cong \begin{pmatrix}
  {\psi_{1}} & *\\ 0&
  {\psi_{2}}\omega^{2-k}\epsilon^{-1} \end{pmatrix} $$with $\psi_{1}$,
$\psi_{2}$ unramified lifts of
$\overline{\psi}_{1}|_{\Gal(\overline{E_x}/E_x)}$ and
$\overline{\psi}_{2}|_{\Gal(\overline{E_x}/E_x)}$ respectively, and
$\omega$ the Teichm\"uller lift of $\epsilon$.

\begin{corollary}\label{cor:potential modularity in weight 0 for f} Let $\mathcal{N}$ be a finite
  set of even positive integers. Then there is a Galois totally real
  extension $F'/E$ and a place $w_{q}|q$ of $F'$ such
  that \begin{itemize}
	\item for any $n\in\mathcal{N}$, $\rho_{n}|_{\Gal(\Qbar/F')}$ is automorphic of weight 0 and type $\{\Sp_{n}(1)\}_{\{w_{q}\}}$,
	\item every prime of $E$ dividing $l$ is unramified in
          $F'$ (so that $l$ is unramified in $F'$), and 
	\item $F'$ is linearly disjoint from $\Qbar^{\ker({\rhobar_{f,\lambda}})}$ over $\Q$.
        \end{itemize}Let $\iota:\Qlbar\isoto\C$, and for
        $n\in\mathcal{N}$ let $\pi_{n}$ be the RAESDC representation
        of $\GL_n(\A_{F'})$ with
        $r_{l,\iota}(\pi_{n})\cong\rho_{n}|_{\Gal(\Qlbar/F')}$. If $k=2$ then $\pi_{n,x}$
        is unramified for each $x|l$, and if $k>2$ then for each place
        $x|l$ of $F'$, $\pi_{n,x}$ is a principal series
        representation
        $\nind_{B_{n}(F'_x)}^{\GL_{n}(F'_x)}(\chi_{1},\dots,\chi_{n})$
        with
        $\iota^{-1}\chi_{i}\circ\Art_{F'_x}^{-1}|_{I_x}=\omega^{(i-1)(2-k)}$ and 
 $v_l(\iota^{-1}\chi_i(l))=[F'_x:\Q_l]\left(i-1+\frac{1-n}{2}\right)$, where $v_l$ is the $l$-adic valuation on $\Qlbar$ with $v_l(l)=1$.
	
\end{corollary}
\begin{proof}This is a straightforward consequence of Theorem
  \ref{thm:potential modularity in weight 0 for sym of ab var}. The
  only part that needs to be checked is the assertion about the form
  of $\pi_{n,x}$ for $x|l$ when $k>2$. Without loss of generality, we
  may assume that $2\in\mathcal{N}$. Note firstly that any principal series
  representation of the given form is irreducible, so that we need
  only check
  that $$\iota^{-1}\rec(\pi_{n,x})=\bigoplus_{i=1}^{n}
  \omega^{(i-1)(2-k)}\alpha_i, $$where $\alpha_i$ is an unramified
  character with $v_l(\alpha_i(l))=[F'_x:\Q_l]\left(i-1+\frac{1-n}{2}\right)$. By Definition
  \ref{defn:sym of abelian variety is automorphic} and Theorem
  \ref{thm:potential modularity in weight 0 for sym of ab var} we see
  that $\rec(\pi_{n,x})=\Sym^{n-1}\rec(\pi_{2,x})$, so
  it suffices to establish the result in the case $n=2$, or rather
  (because of the compatibility of $\rec$ with twisting) it suffices
  to check the corresponding result for $\WD_{v}(A,i)$ at places
  $v|l$. This is now an immediate consequence of local-global
  compaitibility, and follows at once from, for example, Lemma B.4.1 of
  \cite{cdt}, together with the computations of the Weil-Deligne
  representations associated to characters in section B.2 of
  \emph{loc. cit.}
	
\end{proof}

\section{Changing weight}\label{hida theory}\subsection{}We now explain how to deduce
from the results of the previous section that
$\Sym^n\rhobar_{f,\lambda}$ is potentially automorphic of the correct
weight (that is, the weight of the conjectural automorphic representation corresponding to $\Sym^n\rho_{f,\lambda}$), rather than potentially automorphic of weight $0$. We
accomplish this as a basic consequence of Hida theory; note
that we simply need a congruence, rather than a result about families,
and the result follows from a straightforward combinatorial argument. This result
is certainly known to the experts, but as we have been unable to find
a reference which provides the precise result we need, we present a
proof in the spirit of the arguments of \cite{taylorthesis}.

\subsection{}\label{chenevier section}
For each $n$-tuple of integers $a=(a_1,\dots,a_n)$ with $a_1\geq\dots\geq a_n$ there is an irreducible representation of the algebraic group $\GL_n$ defined over $\Ql$, with highest weight (with respect to the Borel subgroup of upper-triangular matrices) given by $$\diag(t_1,\dots,t_n)\mapsto\prod_{i=1}^n t_i^{a_i}.$$ We will need an explicit model of this representation, for which we follow section 2 of  \cite{MR2075765}. 

Let $K$ be an algebraic extension of $\Ql$, $N$ the subgroup of $\GL_n(K)$
consisting of upper triangular unipotent matrices, $\overline{N}$
the subgroup of lower triangular unipotent matrices, and $T$ the subgroup of diagonal matrices. Let
$R:=K[\GL_n]=K[\{X_{i,j}\}_{1\leq i,j\leq n},\det(X_{i,j})^{-1}]$. We
have commuting natural actions of $\GL_n(K)$ on $R$ by left and right
multiplication. For an element $g\in\GL_n(K)$ we denote these actions
by $g_l$ and $g_r$ respectively, so that if we let $M$ denote the
matrix $(X_{i,j})_{i,j}\in M_n(R)$, we have $$(g_l.X)_{i,j}=g^{-1}M$$
and $$(g_r.X)_{i,j}=Mg.$$ If $(t_1,\dots,t_n)\in\Z^n$, we say that an
element $f\in R$ is of left weight $t$ (respectively of right weight
$t$) if for all $d\in T$ we have $d_lf=t^{-1}(d)f$ (respectively $d_rf=t(d)f$) where $$t(\diag(x_1,\dots,x_n))=\prod_{i=1}^nx_i^{t_i}.$$ 

For each $1\leq i\leq n$ and each $i$-tuple $j=(j_1,\dots,j_i)$, $1\leq j_1<\dots<j_i\leq n$, we let $Y_{i,j}$ be the minor of order $i$ of $M$ obtained by taking the entries from the first $i$ rows and columns $j_1$,\dots,$j_i$. Let $R^{\overline{N}}$ denote the subalgebra of $R$ of elements fixed by the $g_l$-action of $\overline{N}$; it is easy to check that $Y_{i,j}\in R^{\overline{N}}$. Because $T$ normalises $\overline{N}$ it acts on $R^{\overline{N}}$ on the left, and we let $R_t^{\overline{N}}$ be the sub $K$-vector space of elements of left weight $t$; this has a natural action of $\GL_n(K)$ induced by $g_r$.

\begin{prop}\label{sec:weight-changing-chenevier-wts-N-invts}
    Suppose that $t_1\geq\dots\geq t_n$. Then $R^{\overline{N}}_t$ is a model of the irreducible algebraic representation of $\GL_n(K)$ of highest weight $t$. Furthermore, it is generated as a $K$-vector space by the monomials in $Y_{i,j}$ of left weight $t$, and a highest weight vector is given by the unique monomial in $Y_{i,j}$ of left and right weight $t$.
  \end{prop}
  \begin{proof}
    This follows from Proposition 2.2.1 of \cite{MR2075765}.
  \end{proof}

  Assume that in fact $t_1\geq\dots\geq t_n\geq 0$, and let $X_t$
  denote the free $\bigO_K$-module with basis the monomials in
  $Y_{i,j}$ of left weight $t$. By Proposition
  \ref{sec:weight-changing-chenevier-wts-N-invts}, $X_t$ is a
  $\GL_n(\bigO_K)$-stable lattice in $R^{\overline{N}}_t$. Let $T^+$
  be the submonoid of $T$ consisting of elements of the
  form $$\diag(l^{b_1},\dots,l^{b_n})$$ with $b_1\geq\dots\geq b_n\geq
  0$; then $X_t$ is certainly also stable under the action
  of $T^+$. Let $\alpha=\diag(l^{b_1},\dots,l^{b_n})\in T^+$. We wish
  to determine the action of $\alpha$ on $X_t$.

\begin{lemma}\label{lem:lowest weight vector divisibility by l}
  If $Y\in X_t$ is a monomial in the $Y_{i,j}$, then
  $\alpha(Y)\subset l^{\sum_{i=1}^nb_it_{n+1-i}}X_t$. If in fact
  $b_1>\dots>b_n$ then $\alpha(Y)\subset
  l^{1+\sum_{i=1}^nb_it_{n+1-i}}X_t$ unless $Y$ is the unique lowest
  weight vector.
\end{lemma}
\begin{proof}
  If $Y$ has (right) weight $(v_1,\dots,v_n)$, then
  $\alpha(Y)=l^{\sum_{i=1}^nb_iv_i}Y$. The unique lowest weight vector has
  weight $(t_n,\dots,t_1)$, so it suffices to prove that for any other
  $Y$ of weight $(v_1,\dots,v_n)$ which occurs in
  $R^{\overline{N}}_t$, the quantity $\sum_{i=1}^nb_iv_i$ is at least
  as large, and is strictly greater if $b_1>\dots>b_n$. However, by
  standard weight theory we know that we may obtain $(v_1,\dots,v_n)$
  from $(t_n,\dots,t_1)$ by successively adding vectors of the form
  $(0,\dots,1,0\dots,0,-1,0,\dots,0)$, and it is clear that the
  addition of such a vector does not decrease the sum, and in fact
  increases it if $b_1>\dots>b_n$, as required.
\end{proof}

We define a new action of $T^+$ on $X_t$, which we denote by $\cdot{}_{twist}$, by multiplying the natural action of $\diag(l^{b_1},\dots,l^{b_n})$ by $l^{-\sum_{i=1}^nb_it_{n+1-i}}$; this is legitimate by Lemma \ref{lem:lowest weight vector divisibility by l}.
\subsection{}

Fix for the rest of this section a choice of isomorphism
$\iota:\Qlbar\isoto\C$. Assume for the rest of this
section that $F'$ is a totally real field in which each $l$ is
unramified, and
$\pi'$ is an RAESDC representation of $\GL_n(\A_{F'})$ of weight 0 and
type $\{\Sp_n(1)\}_{\{w_q\}}$ for some place $w_q|q$ of ${F'}$, with
$(\pi')^\vee=\chi\pi'$. Suppose furthermore that there is an integer
$k>2$ such that
  \begin{itemize}
  \item 
 
for each place $x|l$, $\pi'_x$ is a principal series
 $\nind_{B_{n}(F_x)}^{\GL_{n}(F_x)}(\chi_{1},\dots,\chi_{n})$
        with  $v_l(\iota^{-1}\chi_i(l))=[F'_x:\Q_l]\left(i-1+\frac{1-n}{2}\right)$ and
        $\iota^{-1}\chi_{i}\circ\Art_{F_x}^{-1}|_{I_x}=\omega^{(i-1)(2-k)}$. \end{itemize}
 (See Corollary \ref{cor:potential modularity in weight 0
   for f} for an example of such a representation.) We transfer to a
 unitary group, following section 3.3 of \cite{cht}. Firstly, we make
 a quartic totally real Galois extension $F/{F'}$, linearly disjoint from
 $\Qbar^{\ker \overline{r}_{l,\iota}(\pi)}$ over $\Q$, such that $w_q$
 and all primes dividing $l$ split in $F$. Let $S(B)$ be the set of
 places of $F$ lying over $w_q$. Let $E$ be a imaginary quadratic
 field in which $l$ and $q$ split, such that $E$ is
 linearly disjoint from $\Qbar^{\ker \overline{r}_{l,\iota}(\pi)}$
 over $\Q$. Let $M=FE$. Let $c$ denote the nontrivial element of
 $\Gal(M/F)$. Let $S_l$ denote the places of $F$ dividing $l$, and let
 $\tilde{S}_l$ denote a set of places of $M$ dividing $l$ such that
 the natural map $\tilde{S}_l\to S_l$ is a bijection. If $v|l$ is a
 place of $F$ then we write $\tilde{v}$ for the corresponding place in
 $\tilde{S}_l$.

\begin{lemma}
  There is a finite order character $\phi:M^\times\backslash \A_M^\times\to\C^\times$ such that
  \begin{itemize}
  \item $\phi\circ N_{M/F}=\chi\circ N_{M/F}$, and
\item $\phi$ is unramified at all places lying over $S(B)$ and at all places in $\tilde{S}_l$.
  \end{itemize}

\end{lemma}
\begin{proof}
By Lemma 4.1.1 of \cite{cht} (or more properly its proof, which shows that the character produced may be arranged to have finite order) there is a finite order character $\psi:M^\times\backslash \A_M^\times\to\C^\times$ such that for each $v\in S_l$, $\psi|_{{M_{\tilde{v}}}^\times}=1$ and $\psi|_{{M_{c\tilde{v}}}^\times}=\chi|_{F^\times_v}$, and such that $\psi$ is unramified at each place in $S(B)$. It now suffices to prove the result for the character  $\chi(\psi|_{\A_{F}^\times})^{-1}$, which is unramified at $S(B)\cup {S}_l$, and the result now follows from Lemma 4.1.4 of \cite{cht}.
\end{proof}
Now let $\pi=\pi'_M\otimes\phi$, which is an RACSDC representation of
$\GL_n(\A_M)$, satisfying:
\begin{itemize}
\item $\pi$ has weight $0$.
\item $\pi$ has type $\{\Sp_n(1)\}_{w|w_q}$.
\item for each place $x\in \tilde{S}_l$, $\pi_x$ is a principal series
 $\nind_{B_{n}(M_x)}^{\GL_{n}(M_x)}(\chi_{1},\dots,\chi_{n})$
        with  $v_l(\iota^{-1}\chi_i(l))=[F'_x:\Q_l]\left(i-1+\frac{1-n}{2}\right)$ and
        $\iota^{-1}\chi_{i}\circ\Art_{M_x}^{-1}|_{I_x}=\omega^{(i-1)(2-k)}$ 
 with $k>2$.
\end{itemize}

\subsection{}Choose a division algebra $B$ with centre $M$ such that
\begin{itemize}
\item $B$ splits at all places not dividing a place in $S(B)$.
\item If $w$ is a place of $M$ lying over a place in $S(B)$, then $B_w$ is a division algebra.
\item $\dim_MB=n^2$.
\item $B^{op}\cong B\otimes_{M,c}M$.
\end{itemize}

For any involution $\ddag$ on $B$ with $\ddag|_M=c$, we may define a reductive algebraic group $G_{\ddag}/F$ by $$G_{\ddag}(R)=\{g\in B\otimes_{F}R: g^{\ddag\otimes 1}g=1\}$$ for any $F$-algebra $R$. Because $[F:\Q]$ is divisible by $4$ and $\# S(B)$ is even, we may (by the argument used to prove Lemma 1.7.1 of \cite{ht}) choose $\ddag$ such that
\begin{itemize}
\item If $v\notin S(B)$ is a finite place of $F$ then $G_{\ddag}(F_v)$ is quasi-split, and
\item If $v|\infty$, $G_\ddag(F_v)\cong U(n)$.
\end{itemize}
Fix such a choice of $\ddag$, and write $G$ for $G_{\ddag}$. We wish to work with algebraic modular forms on $G$; in order to do so, we need an integral model for $G$. We obtain such a model by fixing an order $\bigO_B$ in $B$ such that $\bigO_B^{\ddag}=\bigO_B$ and $\bigO_{B,w}$ is a maximal order for all primes $w$ which are split over $M$ (see section 3.3 of \cite{cht} for a proof that such an order exists). We now regard $G$ as an algebraic group over $\bigO_{F}$, by defining $$G(R)=\{g\in \bigO_B\otimes_{\bigO_{F}}R: g^{\ddag\otimes 1} g=1\}$$ for all $\bigO_{F}$-algebras $R$. 

We may  identify $G$ with $\GL_n$ at places not in $S(B)$ which split in $M$ in the following way. Let $v\notin S(B)$ be a place of $F$ which splits in $M$. Choose an isomorphism $i_v:\bigO_{B,v}\isoto M_n(\bigO_{M_v})$ such that $i_v(x^{\ddag})={}^ti_v(x)^c$ (where ${}^t$ denotes matrix transposition). Choosing a prime $w|v$ of $M$ gives an isomorphism
\begin{align*}
  i_w:G(F_v)&\isoto\GL_n(M_w)\\
i_v^{-1}(x,{}^tx^{-c})&\mapsto x.
\end{align*}This identification satisfies $i_wG(\bigO_{F,v})=\GL_n(\bigO_{M,w})$. Similarly, if $v\in S(B)$ then $v$ splits in $M$, and if $w|v$ then we obtain an isomorphism $$i_w:G(F_v)\isoto B_w^\times$$ with $i_wG(\bigO_{F,v})=\bigO_{B,w}^\times$.

Now let $K=\Qlbar$. Write $\bigO$ for the ring of
integers of $K$, and $k$ for the residue field $\Flbar$.

Let $I_l=\Hom(F,K)$, and let $\tilde{I}_l$ be the subset of elements
of $\Hom(M,K)$ such that the induced place of $M$ is in
$\tilde{S}_l$. Let $a\in(\Z^n)^{\Hom(M,K)}$; we assume that
\begin{itemize}
\item $a_{\tau,1}\geq\dots\geq a_{\tau,n}\geq 0$ if $\tau\in\tilde{I}_l$, and
\item $a_{\tau c,i}=-a_{\tau,n+1-i}$.
\end{itemize}
Consider the constructions of section \ref{chenevier section} applied
to our choice of $K$. Then we have an
$\bigO$-module $$Y_a=\otimes_{\tau\in\tilde{I}_l}X_{a_\tau}$$which
has a natural action of $G(\bigO_{F,l})$, where $g\in G(\bigO_{F,l})$
acts on $X_{a_\tau}$ by $\tau(i_\tau g_\tau)$. From now on, if $v|l$
is a place of $F$, we will identify $G(\bigO_{F_v})$ with
$\GL_n(\bigO_{M_{\tilde{v}}})$ via the map $i_{\tilde{v}}$ without
comment.

We say that an open compact subgroup $U\subset G(A_F^\infty)$ is
sufficiently small if for some place $v$ of $F$ the projection of $U$
to $G(F_v)$ contains no nontrivial elements of finite order. Assume
from now on that $U$ is sufficiently small, and in addition that we
may write $U=\prod_v U_v$, $U_v\subset G(\bigO_{F_v})$, such that
\begin{itemize}
\item if $v\in S(B)$ and $w|v$ is a place of $M$, then $i_w(U_v)=\bigO_{B,w}^\times$, and
\item if $v|l$ then $U_v$ is the Iwahori subgroup of  matrices which
  are upper-triangular mod $l$.
\end{itemize}
If $v|l$, let $U'_v$ denote the pro-$l$ subgroup of $U_v$ corresponding to the group of matrices which are (upper-triangular) unipotent mod $l$, and let $$\chi_v:U_v/U_v'\to\bigO^\times$$ be a character. Let $\chi=\otimes\chi_v:\prod_{v|l}U_v\to\bigO^\times$, and write $$Y_{a,\chi}=Y_a\otimes_\bigO\chi,$$a $\prod_{v|l}U_v$-module.

Let $A$ be an $\bigO$-algebra. Then we define the space of algebraic
modular forms $$S_{a,\chi}(U,A)$$ to be the space of
functions $$f:G(F)\backslash G(\A_{F}^\infty)\to A\otimes_\bigO Y_{a,\chi}$$
satisfying $$f(gu)=u^{-1}f(g)$$ for all $u\in U$, $g\in
G(\A_{F}^\infty)$, where the action of $U$ on $A\otimes_\bigO Y_{a,\chi}$ is
inherited from the action of $\prod_{v|l}U_v$ on $Y_{a,\chi}$. Note that
because $U$ is sufficiently small we
have $$S_{a,\chi}(U,A)=S_{a,\chi}(U,\bigO)\otimes_\bigO A.$$

More generally, if $V$ is any $U''$-module with $U''$ a sufficiently small compact open subgroup, we define the space of algebraic modular forms $$S(U'',V)$$ to be the space of functions $$f:G(F)\backslash G(\A_{F}^\infty)\to V$$
satisfying $$f(gu)=u^{-1}f(g)$$ for all $u\in U''$, $g\in
G(\A_{F}^\infty)$.

Let $T^+_l$ denote the monoid of  elements of $G(\A_F^\infty)$ which are trivial
outside of places dividing $l$, and at places dividing $l$ correspond
to matrices $\diag(l^{b_1},\dots,l^{b_n})$ with $b_1\geq\dots\geq
b_n\geq 0$. In addition to the action of $U$ on $Y_{a,\chi}$, we can also allow $T^+_l$ to act. We define the action of $T^+_l$ via the action
$\cdot{}_{twist}$ on $X_t$ defined above. This gives us an action of the
monoid $\langle U,T^+_l\rangle$ on $Y_{a,\chi}$. Now suppose that $g$ is an
element of $G(\A_F^\infty)$ with either $g_l\in G(\bigO_{F,l})$ or
$g\in T^+_l$; then we write $$UgU=\coprod_i g_i U,$$a finite union of
cosets, and define a linear map $$[UgU]:S_{a,\chi}(U,A)\to S_{a,\chi}(U,A)$$
by $$([UgU]f)(h)=\sum_i g_if(hg_i).$$ 

We now introduce some notation for Hecke algebras. Let $v$ be a place of $F$ which splits in $M$, and suppose that $v\notin S(B)$ and that $U_v=G(\bigO_{F_v})$ (so, in particular $v\nmid l$). Suppose that $w|v$ is a place of $M$, so that we may regard $G(\bigO_{F_v})$ as $\GL_n(\bigO_{M_w})$ via $i_w$. Then we let $T_w^{(j)}$, $1\leq j\leq n$ denote the Hecke operator given by $$[U\diag(\varpi_w,\dots,\varpi_w,1,\dots,1)U]$$ where $\varpi_w$ is a uniformiser of $M_w$, and there are $j$ occurrences of it in this matrix. We let $\mathbb{T}_{a,\chi}(U,A)$ denote the commutative $A$-subalgebra of $\End(S_{a,\chi}(U,A))$ generated by the operators $T_w^{(j)}$ and $(T^{(n)}_w)^{-1}$ for all $w$, $j$ as above. Note that $\mathbb{T}_{a,\chi}(U,A)$ commutes with $[UgU]$ for all $g\in T^+_l$. More generally, let $\mathbb{T}(U)$ denote the polynomial ring over $\bigO$ in the formal variables $T_w^{(j)}$ and $(T_w^{n})^{-1}$, which we may think of as acting on $S_{a,\chi}(U,A)$ via the obvious map $\mathbb{T}(U)\to\mathbb{T}_{a,\chi}(U,A)$.

We also wish to consider the Hecke operator $U_l=[UuU]$, where $u\in T^+_l$ has $u_v=\diag(l^{n-1},\dots,l,1)$ for each $v|l$. As usual, we can define a Hida idempotent $$e_l=\lim_{n\to\infty}U_l^{n!},$$which has the property that $U_l$ is invertible on $e_lS_{a,\chi}(U,\bigO)$ and is topologically nilpotent on $(1-e_l)S_{a,\chi}(U,\bigO)$. We write $$S_{a,\chi}^{ord}(U,A):=e_l S_{a,\chi}(U,A).$$ 

Let $a\in(\Z^n)^{\Hom(M,K)}$ be a weight, and let $\chi_a=\otimes_{v|l}\chi_{a,v}$, where
  $\chi_{a,v}:U_v/U'_v\cong((\bigO_{M_v}/\mathfrak{m}_{M_v})^\times)^n\to\bigO^\times$ is given by the character $(x_1,\dots,x_n)\mapsto\prod_{{\tau}}\prod_i
  {\tau}(\tilde{x}_i)^{a_{\tilde{v},n+1-i}}$, where $\tilde{x}_i$ is the Teichm\"uller lift of $x_i$, and the product is over the embeddings ${\tau}\in \tilde{I}_l$ which give rise to $v$.

The main lemma we require is the following.

\begin{lemma}\label{lem:hida theory}Let $a$ be a weight. Then there is a
  $\mathbb{T}(U)$-equivariant isomorphism $$S_{a,\chi}^{ord}(U,k)\to
  S_{0,\chi\chi_a}^{ord}(U,k).$$
\end{lemma}

\begin{proof}
  Note firstly that there is a natural projection map $j$ from
  $Y_{a,\chi}$ to the $\bigO$-module given by the tensor product $Z_{a,\chi}$ of the lowest weight
  vectors. This is a map of
  $\prod_{v|l}U_v$-modules, and by Lemma \ref{lem:lowest weight
    vector divisibility by l} we see that $j$ induces an
  isomorphism $$u\cdot{}_{twist}Y_{a,\chi}\otimes_\bigO k\to
  u\cdot{}_{twist}Z_{a,\chi}\otimes_\bigO k.$$ Note also that by definition we have an
  isomorphism of $\langle U,T_l^+\rangle$-modules $Z_{a,\chi}\to
  Y_{0,\chi\chi_a}$. It thus suffices to prove that the induced map $$j:S_{a,\chi}^{ord}(U,k)\to S^{ord}(U,Z_{a,\chi}\otimes_\bigO k)\text{ }(=S^{ord}_{0,\chi\chi_a}(U,k))$$ is an isomorphism.

  We claim that there is a diagram

$$\xymatrix{S_{a,\chi}(U,k)\ar[r]^j &S(U,Z_{a,\chi}\otimes_\bigO k)\ar[r]^<<<<<{u\cdot{}_{twist}}&  S(U\cap u U u^{-1},u\cdot{}_{twist}Z_{a,\chi}\otimes_\bigO k)\ar[d]^{j^{-1}}\\&\ar[ul]^{\cor} S_{a,\chi}(U\cap u U u^{-1},k)&\ar[l]_>>>>>i S(U\cap u U u^{-1}, u\cdot_{twist}Y_{a,\chi}\otimes_\bigO k)} $$ such that the maps $$\cor\circ i\circ j^{-1}\circ u\cdot{}_{twist}\circ j:S_{a,\chi}(U,k)\to S_{a,\chi}(U,k)$$ and $$j\circ \cor\circ i\circ j^{-1}\circ u\cdot{}_{twist}:S(U,Z_{a,\chi}\otimes_\bigO k)\to S(U,Z_{a,\chi}\otimes_\bigO k)$$ are both given by $U_l$. Since $U_l$ is an isomorphism on $S_{a,\chi}^{ord}(U,k)$, the result will follow.

  In fact, the construction of the diagram is rather straightforward. The maps $j$, $j^{-1}$ are just the natural maps on the coefficients (note that both are maps of $U$-modules). The map $u\cdot{}_{twist}$ is given by $$(u\cdot{}_{twist}f)(h)=u\cdot_{twist}f(hu).$$ The map $i$ is given by the inclusion of $U$-modules $u\cdot{}_{twist}Y_{a,\chi}\otimes_\bigO k\into Y_{a,\chi}\otimes_\bigO k$. Finally, the map $\cor$ is defined in the following fashion. We may write $$U=\coprod u_i(U\cap u U u^{-1}),$$ and we define $$(\cor f)(h)=\sum u_if(hu_i).$$

The claims regarding the compositions of these maps follow immediately from the observation that $$UuU=\coprod u_i u U.$$

\end{proof}

\subsection{}We now recall some results on tamely ramified principal
series representations of $\GL_n$ from \cite{MR1621409}. Let $L$ be a
finite extension of $\Qp$ for some $p$, and let $\pi_L$ be an
irreducible smooth complex representation of $\GL_n(L)$. Let $I$
denote the Iwahori subgroup of $\GL_n(\bigO_L)$ consisting of matrices
which are
upper-triangular mod $\mathfrak{m}_L$, and let $I_1$ denote its Sylow pro-$l$
subgroup. Let $l$ be the residue field of $L$, and let $\varpi_L$
denote a uniformiser of $L$. Then there is a natural isomorphism
$I/I_1\cong(l^\times)^n$. If $\chi=(\chi_1,\dots,\chi_n):
(l^\times)^n\to\C^\times$ is a character, then we let $\pi_L^{I,\chi}$
denote the space of vectors in $\pi_L$ which are fixed by $I_1$ and
transform by $\chi$ under the action of $I/I_1$. The space
$\pi_L^{I,\chi}$ has a natural action of the Hecke algebra
$\mathcal{H}(I,\chi)$ of compactly supported $\chi^{-1}$-spherical
functions on $\GL_n(L)$. We consider the commutative subalgebra
$\mathbb{T}(I,\chi)$ of $\mathcal{H}(I,\chi)$ generated by double
cosets $[I\alpha I]$ where
$\alpha=\diag(\varpi_L^{b_1},\dots,\varpi_L^{b_n})$ with
$b_1\geq\dots\geq b_n\geq 0$.

If $\chi:(\bigO_L^\times)^n\to\C^\times$ is tamely ramified, then we let $\pi_L^{I,\chi}$ denote $\pi_L^{I,\bar{\chi}}$, where $\bar{\chi}$ is the character $(l^\times)^n\to\C^\times$ determined by $\chi$. Let $\delta$ denote the modulus character of $\GL_n(L)$, so that $$\delta(\diag(a_1,\dots,a_n))=|a_1|^{n-1}|a_2|^{n-3}\dots|a_n|^{1-n}$$ where $|\cdot|$ denotes the usual norm on $L$.

\begin{proposition}\label{facts about principal series}
  \begin{enumerate}
  \item If $\pi_L^I\neq 0$ then $\pi$ is a subquotient of an unramified principal series representation.
  \item If $\pi_L^{I_1}\neq 0$ then $\pi$ is a subquotient of a tamely ramified principal series representation. More precisely, if $\pi_L^{I,\chi}\neq 0$ then $\pi_L$ is a subquotient of a tamely ramified principal series representation $\nind_{B_n(L)}^{\GL_n(L)}(\chi'_1,\dots,\chi'_n)$ with $\chi_i'$ extending $\chi_i$ for each $i$.
  \item If $\pi_L=\nind_{B_n(L)}^{\GL_n(L)}(\chi)$ with $\chi$ tamely ramified, then $$\pi_L^{I,\chi}\cong\oplus_w\chi\delta^{-1/2}$$ as a  $\mathbb{T}(I,\chi)$-module, where the sum is over the elements $w$ of the Weyl group of $\GL_n$ with $\chi^w=\chi$; that is, the double coset $[I\alpha I]$ acts via $(\chi\delta^{-1/2})(\alpha)$ on $\pi_L^{I,\chi}$.
  \end{enumerate}
\end{proposition}
\begin{proof}
  The first two parts follow from Lemma 3.1.6 of \cite{cht} and its proof. All three parts follow at once from Theorem 7.7 and Remark 7.8 of \cite{MR1621409} (which are valid for $\GL_n$ without any restrictions on $L$ - see the proof of Lemma 3.1.6 of \cite{cht}), together with the standard calculation of the Jacquet module of a principal series representation, for which see for example Theorem 6.3.5 of \cite{Casselmannotes} (although note that there is a missing factor of $\delta^{1/2}$ (or rather $\delta_\Omega^{1/2}$ in the notation of \emph{loc. cit.}) in the formula given there). 
\end{proof}

\subsection{} Keep our running assumptions on $\pi$. Suppose that $U=\prod_v U_v$ is
a sufficiently small subgroup of $G(\A_{F})$. Assume further that $U$
has been chosen such that if
$v\notin S(B)$, $v=ww^c$ splits completely in $M$, and $U_v$ is a
maximal compact subgroup of $G(F_v)$, then $\pi_w$ is
unramified. Recall that we have fixed an isomorphism
$\iota:\Qlbar\isoto\C$. There is a maximal ideal
$\mathfrak{m}_{\iota,\pi}$ of $\mathbb{T}(U)$ determined by $\pi$ in
the following fashion. For each place $v=ww^c$ as above the Hecke operators
$T^{(i)}_w$ act via scalars $\alpha_{w,i}$ on
$(\pi_w)^{\GL_n(\bigO_{M_w})}$. The $\alpha_{w,i}$ are all algebraic
integers, so that
$\iota^{-1}(\alpha_{w,i})\in \bigO$. Then
$\mathfrak{m}_{\iota,\pi}$ is the maximal ideal of $\mathbb{T}(U)$
containing all the $T^{(i)}_w-\iota^{-1}(\alpha_{w,i})$. Let
$\sigma_k\in(\Z^n)^{\Hom(M,K)}$ be the weight determined by
$(\sigma_k)_{\tau,i}=(k-2)(n-i)$ for each $\tau\in\tilde{I}_l$.

\begin{lemma}\label{moving to unitary group}
  Suppose that $\pi$ is a RACSDC representation of $\GL_n(\A_M)$ of
  weight 0 and type $\{\Sp_n(1)\}_{S(B)}$. Suppose that for each place
  $x\in \tilde{S}_l$, $\pi_x$ is a principal series
  $\nind_{B_{n}(M_x)}^{\GL_{n}(M_x)}(\chi_{x,1},\dots,\chi_{x,n})$
  with
  $\iota^{-1}\chi_{x,i}\circ\Art_{F_x}^{-1}|_{I_x}=\omega^{(i-1)(2-k)}$. Then
  there is a sufficiently small compact open subgroup $U$ of
  $G(\A_{F})$ such that $U$ satisfies the requirements above (in
  particular, $U=\prod_v U_v$ where $U_v$ is an Iwahori subgroup of
  $\GL_n(F_v)$ for each $v|l$) and
  $S_{0,\chi_{\sigma_k}}(U,\bigO)_{\mathfrak{m}_{\iota,\pi}}\neq 0$. If we assume
  furthermore that $v_l(\iota^{-1}\chi_{x,i}(l))=[M_x:\Q_l]\left(i-1+\frac{1-n}{2}\right)$ for all
  $i$ (and all $x\in\tilde{S}_l$) then
  $S_{0,\chi_{\sigma_k}}^{ord}(U,\bigO)_{\mathfrak{m}_{\iota,\pi}}\neq 0$.
\end{lemma}
\begin{proof}
  This is a consequence of Proposition 3.3.2 of \cite{cht}. The only
  issues are at places dividing $l$ and places in $S(B)$. For the
  latter, it is enough to note that under the Jacquet-Langlands
  correspondence, $\Sp_n(1)$ corresponds to the trivial
  representation. For the first part, we also need to check that at each  place $x\in\tilde{S}_l$, $\pi_{x}^{I_x,\chi_x}\neq 0$, where $I_x$ is the standard Iwahori subgroup of $\GL_n(M_x)$, and $\chi_x=(\chi_{x,1},\dots,\chi_{x,n})$. This follows at once from Proposition \ref{facts about principal series}. 

  For the second part, we must check in addition that if the Hecke
  operator $[I_xu_xI_x]$ (where $u_x=\diag(l^{n-1},\dots,1)$) acts via
  the scalar $\alpha_x$ on $\pi_{x}^{I_x,\chi_x}$, then
  $\iota^{-1}(\alpha_x)$ is an $l$-adic unit. This is straightforward;
  by Proposition \ref{facts about principal series}(3),
  $\alpha_x=\chi_x(u)\delta^{-1/2}(u)$. Thus
  \begin{align*}v_l(\iota^{-1}(\alpha_x))&=v_l(\iota^{-1}(\chi_x(u)\delta^{-1/2}(u)))\\
    &= \sum_{i=1}^n(n-i)
    v_l(\iota^{-1}\chi_{x,i}(l))+\sum_{i=1}^n(n-i)v_l((l^{-[M_x:\Ql]})^{-(n+1-2i)/2)})
    \\&=\sum_{i=1}^n
    (n-i)([M_x:\Q_l]\left(i-1+\frac{1-n}{2}\right))+\sum_{i=1}^n[M_x:\Ql](n-i)(n+1-2i)/2\\&=\frac{[M_x:\Ql]}{2}\sum_{i=1}^n(n-i)((2i-1-n)+(n+1-2i))\\&=0,
  \end{align*}as required.
\end{proof}

\begin{lemma}\label{lem:existence of the ordinary unramified representation}
Keep (all) the assumptions of Lemma \ref{moving to unitary
  group}. Then there is an RACSDC representation $\pi''$ of
$\GL_n(\A_M)$ of weight $\iota_*\sigma_k$, type
$\{\Sp_n(1)\}_{\{S(B)\}}$ and with $\pi''_l$ unramified such that  $\overline{r}_{l,\iota}(\pi'')\cong\bar{r}_{l,\iota}(\pi)$.
\end{lemma}
\begin{proof}
  This is essentially a consequence of Lemma \ref{moving to unitary
    group}, Lemma \ref{lem:hida theory}, and Proposition \ref{facts
    about principal series}, together with Proposition 3.3.2 of
  \cite{cht}. Indeed, Lemma \ref{lem:hida theory}  and Lemma
  \ref{moving to unitary group} show that
  $S_{\sigma_k,1}^{ord}(U,\bigO)_{\mathfrak{m}_{\iota,\pi}}\neq
  0$, which by Proposition \ref{facts about
    principal series}(1) and Proposition 3.3.2 of \cite{cht} gives us a $\pi''$
  satisfying all the properties we claim, except that we only know that
  for each $x|l$, $\pi''_x$ is a subquotient of an unramified principal
  series representation. We claim that this unramified principal
  series is irreducible, so that $\pi''_x$ is unramified. To see this, note that the fact that we know that $S_{\sigma_k,1}^{ord}(U,\bigO)_{\mathfrak{m}_{\iota,\pi}}\neq 0$ (rather than merely $S_{\sigma_k,1}(U,\bigO)_{\mathfrak{m}_{\iota,\pi}}\neq 0$) means that we can choose $\pi''$ so that for each $x\in\tilde{S}_l$, $\pi''_x$ is a subquotient of an unramified principal series representation $\nind_{B_{n}(M_x)}^{\GL_{n}(M_x)}(\chi_{x,1},\dots,\chi_{x,n})$ with $$v_l(\iota^{-1}\chi_{x,i}(l))=[M_x:\Ql]\left((i-1)(k-1)+(1-n)/2\right)$$(this follows from the comparison of the Hecke actions on $(\pi''_x)^{I_x}$ and  $S_{\sigma_k,1}(U,\bigO)$, noting that the latter action is defined in terms of $\cdot_{twist}$). Now, if the principal series  $\nind_{B_{n}(M_x)}^{\GL_{n}(M_x)}(\chi_{x,1},\dots,\chi_{x,n})$ were reducible, there would be $i$, $j$ with $\chi_{x,i}=\chi_{x,j}|\cdot|$, so that $\chi_{x,i}(l)l^{[M_x:\Ql]}=\chi_{x,j}(l)$, which is a contradiction because $k>2$. The result follows.
\end{proof}

Combining Corollary \ref{cor:potential modularity in weight 0 for f} with Lemma \ref{lem:existence of the ordinary unramified representation}, we obtain

\begin{proposition}\label{modularity of all of them in weight 0}Let
  $l$ be an odd prime, and let $f$ be a modular form of weight $2\leq
  k <l$ and level coprime to $l$. Assume that $f$ is Steinberg at $q$,
  and that for some place $\lambda|l$ of $\bigO_f$, $f$ is ordinary at
  $\lambda$ and $\rhobar_{f,\lambda}$ is absolutely irreducible. Fix
  an embedding $K_{f,\lambda}\into\Qlbar$. Let $\mathcal{N}$ be a
  finite set of even positive integers. Then there is a Galois totally
  real extension $F/\Q$ and a quadratic imaginary field $E$, together
  with a place $w_q|q$ of $M=FE$ such that if we choose a set $\tilde{S}_l$ of
  places of $M$ consisting of one place above each place of $F$
  dividing $l$, and define $\sigma_k\in(\Z^n)^{\Hom(M,K)}$ by $(\sigma_k)_{\tau,i}=(k-2)(n-i)$, then
  \begin{itemize}
  \item for each $n\in\mathcal{N}$, there is a character
    $\bar{\phi}_n:\Gal(\overline{M}/M)\to\Flbar^\times$ which is
    unramified at all places in $\tilde{S}_l$, which satisfies $$\bar{\phi}_n\bar{\phi}_n^c=(\epsilon\det\rhobar_{f,\lambda}\otimes\Flbar)^{1-n}|_{\Gal(\overline{M}/M)}$$
    and
    $(\Sym^{n-1}\rhobar_{f,\lambda}\otimes\Flbar)|_{\Gal(\overline{M}/M)}\otimes\bar{\phi}_n$
    is automorphic of weight $\sigma_k$ and type $\{\Sp_n(1)\}_{\{w_q\}}$.
  \item $l$ is unramified in $M$.
  \item $M$ is linearly disjoint from $\Qbar^{\ker(\rhobar_{f,\lambda})}$ over $\Q$.
  \end{itemize}

\end{proposition}

\section{Potential automorphy}\label{sec:potential automorphy}\subsection{}Assume as before that $f$
is a cuspidal newform of level $\Gamma_1(N)$, weight $k\geq 2$, and
nebentypus $\chi_f$. Let $\pi(f)$ be the RAESDC representation of
$\GL_2(\A_\Q)$ corresponding to $f$. We will think of $\chi_f$ as an
automorphic representation of $\GL_1(\A_\Q)$, and write
$\chi_f=\otimes_p\chi_{f,p}$. We now define what we mean by the claim that the
symmetric powers of $f$ are potentially automorphic. If $F$ is a
totally real field and $v|p$ is a place of $F$, we write
$\rec(\pi_{f,p})|_{F_{v}}$ for the restriction of the Weil-Deligne
representation $\rec(\pi_{f,p})$ to the Weil group of $F_{v}$. Then we
say that $\Sym^{n-1}f$ is potentially automorphic over $F$ if there is
an RAESDC representation $\pi_{n}$ of $\GL_{n}(\A_{F})$ such that for
all primes $p$ and all places $v|p$ of $F$ we
have $$\rec(\pi_{n,v})=\Sym^{n-1}(\rec(\pi_{f,p})|_{F_{v}}). $$ By a
standard argument (see for example section 4 of \cite{hsbt}) this is
equivalent to asking that
$\Sym^{n-1}\rho_{f,\lambda}|_{\Gal(\overline{F}/F)}$ be automorphic for
some place (equivalently for all places)  $\lambda$ of $K_f$.

 Similarly, we may speak of $\Sym^{n-1}f$ being potentially automorphic of a specific weight and type. We then define (for each $n\geq 1$ and each integer $a$) the $L$-series $$L(\chi_f^{a}\otimes\Sym^{n-1}f,s)=\prod_{p}L((\chi_{f,p}^{a}\circ\Art_{\Qp}^{-1})\otimes\Sym^{n-1}\rec(\pi_{f,p}),s+(1-n)/2). $$

We now normalise the $L$-functions of RAESDC automorphic representations to agree with those of their corresponding Galois representations. Specifically, if $\pi$ is an RAESDC representation of $\GL_{n}(\A_{F})$, we define $$L(\pi,s)=\prod_{v\nmid\infty}L(\pi_{v},s+(1-n)/2). $$If $\pi$ is square integrable at some finite place, then for each isomorphism $\iota:\Qbar_{l}\isoto\C$ there is a Galois representation $r_{l,\iota}(\pi)$, and by definition we have \begin{align*}L(\pi,s)&=\prod_{v\nmid\infty}L(\pi_{v}\otimes(|\cdot|\circ\det)^{(1-n)/2},s)\\
&=\prod_{v\nmid\infty}L(\operatorname{rec}(\pi_{v}\otimes(|\cdot|\circ\det)^{(1-n)/2}),s)\\	
&=\prod_{v\nmid\infty}L(r_{l}(\iota^{-1}\pi_{v})^{\vee}(1-n),s)\\
&=L(r_{l,\iota}(\pi),s).
\end{align*}
\begin{thm}\label{thm:actual modularity in the given weights}Suppose that $f$ is a cuspidal newform of level $\Gamma_1(N)$ and weight $k=2$ or $3$. Suppose that $f$ is Steinberg at $q$. Let $\mathcal{N}$ be a finite set of even positive integers. Then there is a Galois totally real field $F$ such that for any $n\in\mathcal{N}$ and any subfield $F'\subset F$ with $F/F'$ soluble, $\Sym^{n-1}f$ is automorphic over $F'$.
\end{thm}
\begin{proof}By Lemma \ref{lem:weight two or three is ordinary} and
  Lemma \ref{lem:modular form large image} we may choose a prime $l>3$
  and a place $\lambda$ of $\bigO_{f}$ lying over $l$ such
  that
	
	\begin{itemize}
		\item $l\nmid N$.
		\item $f$ is ordinary at $\lambda$.
		\item $l>\max(2n+1)_{n\in\mathcal{N}}$.
                \item $\rhobar_{f,\lambda}$ has large image.
	\end{itemize}
By Corollary \ref{modularity of all of them in weight 0} there is an
embedding $K_{f,\lambda}\into\Qlbar$, a Galois totally real extension
$F/\Q$ and a quadratic imaginary field $E$, together with a place
$w_q|q$ of $M=FE$ such that if we choose a set $\tilde{S}_l$ of
  places of $M$ consisting of one place above each place of $F$
  dividing $l$, and define $\sigma_k\in(\Z^n)^{\Hom(M,K)}$ by $(\sigma_k)_{\tau,i}=(k-2)(n-i)$, then
  \begin{itemize}
  \item  for each $n\in\mathcal{N}$, there is a character
    $\bar{\phi}_n:\Gal(\overline{M}/M)\to\Flbar^\times$ which is
    unramified at all places in $\tilde{S}_l$ and satisfies $$\bar{\phi}_n\bar{\phi}_n^c=(\epsilon\det\rhobar_{f,\lambda}\otimes\Flbar)^{1-n}|_{\Gal(\overline{M}/M)},$$
    and
    $(\Sym^{n-1}\rhobar_{f,\lambda}\otimes\Flbar)|_{\Gal(\overline{M}/M)}\otimes\bar{\phi}_n$
    is automorphic of weight $\sigma_k$ and type $\{\Sp_n(1)\}_{\{w_q\}}$.
  \item $l$ is unramified in $M$.
  \item $M$ is linearly disjoint from $\Qbar^{\ker(\rhobar_{f,\lambda})}$ over $\Q$.
  \end{itemize} 
  Fix $n\in\mathcal{N}$, and let
  $\rho:=\Sym^{n-1}\rho_{f,\lambda}|_{\Gal(\overline{F}/F)}\otimes\Qlbar$.
  There is a crystalline character
  $\chi:\Gal(\overline{F}/F)\to\bigO_{\Qlbar}^\times$ which is
  unramified above $q$ such
  that $$\rho^\vee\cong\rho\chi\epsilon^{n-1};$$in fact, $$\chi=(\epsilon\det\rho_{f,\lambda}\otimes\bigO_{\Qlbar})^{1-n}|_{\Gal(\overline{F}/F)}.$$ By Lemma 4.1.6 of
  \cite{cht} we can choose an algebraic
  character $$\psi:\Gal(\overline{M}/M)\to\bigO_{\Qlbar}^\times$$ such
  that
\begin{itemize}
\item $\chi|_{\Gal(\overline{M}/M)}=\psi\psi^c$,
\item $\psi$ is crystalline,
\item $\psi$ is unramified at each place in $\tilde{S}_l$.
\item $\psi$ is unramified above $q$,
\item $\bar{\psi}=\bar{\phi}_n$.
\end{itemize}
Then $\rho'=\rho|_{\Gal(\overline{M}/M)}\psi$ satisfies $$(\rho')^c\cong (\rho')^\vee\epsilon^{1-n}.$$ We claim that $\rho'$ is automorphic of weight $\sigma_k$, level prime to $l$ and type $\{\Sp_n(1)\}_{\{w_q\}}$. This follows from Theorem 5.2 of \cite{tay06}; we now check the hypotheses of that theorem. Certainly $\bar{\rho}'\cong(\Sym^{n-1}\rhobar_{f,\lambda}\otimes\Flbar)|_{\Gal(\overline{M}/M)}\otimes\bar{\phi}_n$ is automorphic of weight $\sigma_k$, level prime to $l$ and type $\{\Sp_n(1)\}_{\{w_q\}}$. The only non-trivial conditions to check are that:
\begin{itemize}
\item $\overline{M}^{\ker\ad\rhobar'}$ does not contain $M(\zeta_l)$, and
\item The image $\rhobar'(\Gal(\overline{M}/M(\zeta_l)))$ is big in the sense of Definition 2.5.1 of \cite{cht}.
\end{itemize}These both follow from the assumption that $\rhobar_{f,\lambda}$ has large image, the fact that $M$ is linearly disjoint from $\Qbar^{\ker(\rhobar_{f,\lambda})}$ over $\Q$, Corollary 2.5.4 of \cite{cht}, and the fact that $\PSL_2(k)$ is simple if $k$ is a finite field of cardinality greater than 3.

It follows from Lemma 4.3.3 of \cite{cht} that $\rho$ is
automorphic. Then from Lemma 4.3.2 of \cite{cht} we see that for each
$F'$ with $F/F'$ soluble, $\Sym^{n-1}\rho_{f,\lambda}|_{\Gal(\overline{F}/F')}$ is automorphic, as required.
\end{proof}
\begin{corollary}
  \label{actual modularity in the given weights, including for
    potentialy steinberg}Suppose that $f$ is a cuspidal newform of
  level $\Gamma_1(N)$ and weight $k=2$ or $3$. Suppose that $f$ is potentially Steinberg at $q$. Let $\mathcal{N}$ be a finite set of even positive integers. Then there is a Galois totally real field $F$ such that for any $n\in\mathcal{N}$ and any subfield $F'\subset F$ with $F/F'$ soluble, $\Sym^{n-1}f$ is automorphic over $F'$.
\end{corollary}
\begin{proof}
  Let $\theta$ be a Dirichlet character such that $f'=f\otimes\theta$ is Steinberg at $q$. The result then follows from Theorem \ref{thm:actual modularity in the given weights} applied to $f'$.
\end{proof}

\section{The Sato-Tate Conjecture}\label{sec:sato tate}
\subsection{}Let $f$ be a cuspidal newform of level $\Gamma_{1}(N)$, nebentypus $\chi_f$, and weight $k\geq 2$. Suppose that $\chi_f$ has order $m$, so that the image of $\chi_f$ is precisely the group $\mu_{m}$ of $m$-th roots of unity. Let $U(2)_{m}$ be the subgroup of $U(2)$ consisting of matrices with determinant in $\mu_{m}$. For each prime $l\nmid N$, if we write $$X^{2}-a_{l}X+l^{k-1}\chi_f(l)=(X-\alpha_{l}l^{(k-1)/2})(X-\beta_{l}l^{{(k-1)/2}}) $$then (by the Ramanujan conjecture) the matrix $$\begin{pmatrix}\alpha_{l} & 0\\ 0& \beta_{l} \end{pmatrix}$$defines a conjugacy class $x_{f,l}$ in $U(2)_{m}$. A natural generalisation of the Sato-Tate conjecture is

\begin{conj}\label{sato tate conjecture}
	If $f$ is not of CM type, then the conjugacy classes $x_{f,l}$ are equidistributed with respect to the Haar measure on $U(2)_{m}$ (normalised so that $U(2)_{m}$ has measure 1).
\end{conj}
The group $U(2)_{m}$ is compact, and its irreducible representations
are given by $\det^{a}\otimes\Sym^{b}\C^{2}$ for $0\leq a<m$ and
$b\geq 0$. By the corollary to Theorem 2 of section I.A.2 of
\cite{ser67} (noting the different normalisations of $L$-functions in
force there), Conjecture \ref{sato tate conjecture} follows if one knows that for each $b\geq 1$, the functions $L(\chi_f^{a}\otimes\Sym^{b}f,s)$ are holomorphic and non-zero for $\Re s\geq 1+b(k-1)/2$ (the required results for $b=0$ are classical).
\begin{thm}\label{thm:analytic properties of l-functions}Suppose that
  $f$ is a cuspidal newform of level $\Gamma_1(N)$, character $\chi_f$,
  and weight $k=2$ or $3$. Suppose that $\chi_f$ has order $m$. Suppose also that $f$ is potentially Steinberg at $q$ for some prime $q$. Then for all integers $0\leq a<m$, $b\geq 1$ the function $L(\chi_f^{a}\otimes\Sym^{b}f,s)$ has meromorphic continuation to the whole complex plane, satisfies the expected functional equation, and is holomorphic and nonzero in $\Re s\geq 1+b(k-1)/2$.
	
\end{thm}
\begin{proof}The argument is very similar to the proof of Theorem 4.2 of \cite{hsbt}. We argue by induction on $b$; suppose that $b$ is odd, and the result is known for all $1\leq b'< b$. We will deduce the result for $b$ and for $b+1$ simultaneously. Apply Corollary \ref{actual modularity in the given weights, including for potentialy steinberg} with $\mathcal{N}=\{2,b+1\}$. Let $F$ be as in the conclusion of Corollary \ref{actual modularity in the given weights, including for potentialy steinberg}. By Brauer's theorem, we may write $$1=\Sigma_{j}a_{j}\Ind_{\Gal(F/F_{j})}^{\Gal(F/\Q)}\chi_{j} $$where $F\supset F_{j}$ with $F/F_{j}$ soluble, $\chi_{j}$ a character $\Gal(F/F_{j})\to\C^{\times}$, and $a_{j}\in\Z$. Then for each $j$, $\Sym^b{f}$ is automorphic over $F_{j}$, corresponding to an RAESDC representation $\pi_{j}$ of $\GL_{b+1}(\A_{F_{j}})$. In addition, $f$ is automorphic over $F_{j}$, corresponding to an RAESDC representation $\sigma_{j}$ of $\GL_{2}(\A_{F_{j}})$.
	
	Then we have $$L(\chi_f^{a}\otimes\Sym^{b}f,s)=\prod_{j}L(\pi_{j}\otimes(\chi_{j}\circ\Art_{F_{j}})\otimes(\chi_f^{a}\circ N_{F_{j}/\Q}),s)^{a_{j}}, $$
	$$L(\chi_f^{a}\otimes\Sym^{2}f,s)=\prod_{j}L((\Sym^{2}\sigma_{j})\otimes(\chi_{j}\circ\Art_{F_{j}})\otimes(\chi_f^{a}\circ
        N_{F_{j}/\Q}),s)^{a_{j}}, $$and $$L(\chi_f^{a}\otimes\Sym^{b+1}f,s)L(\chi_f^{a+1}\otimes\Sym^{b-1}f,s-k+1)=\prod_{j}L((\pi_{j}\otimes(\chi_{j}\circ\Art_{F_{j}})\otimes(\chi_f^{a}\circ
        N_{F_{j}/\Q}))\times \sigma_{j},s+b(k-1)/2)^{a_{j}}. $$The
        result then follows from the main results of \cite{MR2058610}
        and \cite{MR533066} (in the case $b=1$) together with Theorem 5.1 of \cite{MR610479}.\end{proof}
\begin{corollary}
  \label{sato-tate is true}Suppose that $f$ is a cuspidal newform of level $\Gamma_1(N)$ and weight $k=2$ or $3$. Suppose also that $f$ is potentially Steinberg at $q$ for some prime $q$. Then Conjecture \ref{sato tate conjecture} holds for $f$.
\end{corollary}
Finally, we note that one can make this result more concrete, as one can easily explicitly determine the Haar measure on $U(2)_m$ from that of its finite index subgroup $SU(2)$. One finds that (as already follows from Dirchlet's theorem) the classes $x_{f,l}$ are equidistributed by determinant, and that furthermore the classes with fixed determinant are equidistributed with respect to the natural analogue of the usual Sato-Tate measure. That is, suppose that $\zeta\in\mu_m$, and fix a square root $\zeta^{1/2}$ of $\zeta$. Then any conjugacy class $x_{f,l}$ in $U(2)_m$ with determinant $\zeta$ contains a representative of the form $$\begin{pmatrix}\zeta^{1/2}e^{i\theta_l} & 0\\ 0& \zeta^{1/2}e^{-i\theta_l} \end{pmatrix}$$ with $\theta_l\in[0,\pi]$, and the $\theta_l$ are equidistributed with respect to the measure $\frac{2}{\pi}\sin^2\theta d\theta$.
\bibliographystyle{amsalpha} 

\bibliography{tobybib08} 

\end{document}